\documentclass[10pt]{amsart}
\usepackage{graphicx}
\usepackage{caption}
\usepackage{subfigure}

\usepackage{amsthm}

 \usepackage{amsrefs}
\usepackage{mathtools}
\usepackage{stmaryrd}
\usepackage{framed}
\usepackage{color}
\usepackage{enumerate}
\usepackage{tikz}
\usetikzlibrary{calc}
\usepackage{MnSymbol}
\usepackage{amsmath,mathtools}
\usepackage[noadjust]{marginnote}

\newtheorem{thm}{Theorem}[section]
\newtheorem{prop}[thm]{Proposition}
\newtheorem{lem}[thm]{Lemma}
\newtheorem{cor}[thm]{Corollary}

\theoremstyle{definition}
\newtheorem{definition}[thm]{Definition}

\theoremstyle{definition}
\newtheorem{remark}[thm]{Remark}

\numberwithin{equation}{section}

\newcommand{\R}[1]{{R}_{#1}}
\newcommand{\Rt}[1]{\tilde{{R}}_{#1}}
\newcommand{\ko}{\kappa_0}
\newcommand{\D}{\mathcal{D}}

\newcommand{\hc}{H_\mathbb{C}}
\newcommand{\ac}{A_\mathbb{C}}
\newcommand{\sector}{S}
\def\calS{\mathcal{S}}
\def\calA{\mathcal{A}}
\def\calC{\mathcal{C}}
\def\calD{\mathcal{D}}
\def\bN{{\mathbb N}}
\def\bR{{\mathbb R}}
\def\k0{\kappa_0}

\usepackage{color} 
\newcounter{mnote}
\begin{document}

\title[Time analyticity with higher norm estimates for the 2D NSE]{Time analyticity with higher norm estimates for the 2D Navier-Stokes equations}
\date{\today}
\author[C. Foias]{Ciprian Foias$^{1}$}
\address{$^1$Department of Mathematics\\
Texas A\&M University\\ College Station, TX 77843}
\author[M.S. Jolly]{Michael S. Jolly$^{2}$}
\address{$^2$Department of Mathematics\\
Indiana University\\ Bloomington, IN 47405}
\author[R. Lan]{Ruomeng Lan$^{1}$}
\author[R. Rupam]{Rishika Rupam$^{1}$}
\author[Y. Yang]{Yong Yang$^{1}$}
\author[B. Zhang]{Bingsheng Zhang$^{1}$}
\address{$\dagger$ corresponding author}
\email[C. Foias]{foias@math.tamu.edu}
\email[M. S. Jolly$^\dagger$]{msjolly@indiana.edu}
\email[R. Lan] {rlan@math.tamu.edu}
\email[R. Rupam] {rishika@math.tamu.edu}
\email[Y. Yang] {yytamu@math.tamu.edu}
\email[B. Zhang] {bszhang@math.tamu.edu}
\thanks{This work was supported in part by NSF grant numbers DMS-1109638, and DMS-1109784}





\subjclass[2010]{35Q30,76D05,34G20,37L05, 37L25}
\keywords{Navier-Stokes equations, global attractor, analyticity in time}
\begin{abstract}
This paper establishes bounds on norms of all orders for solutions on the global attractor of the 2D Navier-Stokes equations, complexified in time.  Specifically, for periodic
boundary conditions on $[0,L]^2$, and a force $g\in\calD(A^{\frac{\alpha-1}{2}})$, we show there is a fixed strip 
about the real time axis on which a uniform bound $|A^{\alpha}u|< m_\alpha\nu\kappa_0^\alpha$ holds for each $\alpha \in \bN$.   Here $\nu$ is viscosity, $\k0=2\pi/L$ , and
$m_\alpha$ is explicitly given in terms of $g$ and $\alpha$.  We show that if any element in $\calA$ is in $\D(A^\alpha)$,
 then all of $\calA$ is in $\D(A^\alpha)$, and likewise with  $\D(A^\alpha)$ replaced by $C^\infty(\Omega)$.  We demonstrate the universality of this ``all for one, one for all" law
 on the union of a hierarchal set of function classes.  Finally, we treat the question of
 whether the zero solution can be in the global attractor for a nonzero force by
 showing that if this is so, the force must be in a particular function class.
\end{abstract}
\maketitle

\section{Introduction}
It has been known for nearly half a century that all solutions
in the global attractor $\calA$ of the incompressible 2D Navier-Stokes equations (NSE) can be extended to analytic functions in a uniform strip $\calS$ straddling the real axis 
in the complexified time plane (e.g. \cite{Mas,Iooss,FT79,CF89,T95}).  Moreover, these solutions are uniformly bounded in the norm $|A\cdot|$, where
$A$ is the Stokes operator, $|\cdot|=|\cdot|_{L^2(\Omega)}$, and  $\Omega=[0,L]^2$ is the spatial domain, with periodic boundary conditions.  This is so, at least if the external force $g$ in
the equation is complex-analytic in such a strip.    On the other hand, as remarked in \cite{DFJ05},
if $0 \in \calA$, then by inserting $0$ into the NSE,  one sees that $g$ must be in the domain $\D(A)$, so that
in fact the solution is in $\D(A^2)$.  This in turn implies that $g$ is in $\D(A^2)$, and so on by induction. Thus, $g$ must be in $\D(A^\alpha)$ for any $\alpha\in\bN$.  

Analyticity in time has a number of applications.  It allows for the use of the Cauchy integral
formula in deriving bounds.    It offers a route to proving backward 
uniqueness (see \cite{Ts,HT} for compressible flow, and \cite{ESS} for an application in 3D). The width of the strip of analyticity affects the approximation
of the global attractor by various purely algebraic methods \cite{FT94,FJK96,FJL02}.  
 
In this paper we carry out rather intensive estimates during the inductive process 
described in the opening paragraph to
establish uniform bounds in $|A^\alpha\cdot|$ on a strip $\calS$ of
a specific width $\delta$ for all $\alpha\in\bN$.   This implies 
that if $0\in\calA$, then all elements in $\calA$, as well as $g$, are in $C^\infty(\Omega)$.
We show that the bounds in $|A^\alpha\cdot|$ can be sharpened to some extent, by reducing the
width of the strip according to $\delta_{\alpha+1}=\delta_\alpha/2$.
Moreover, we prove the following ``all for one, one for all" law \cite{Dumas}:
regardless of whether $0\in \calA$: if any element in $\calA$ is in $\D(A^\alpha)$,
 then all of $\calA$ is in $\D(A^\alpha)$, and likewise with  $\D(A^\alpha)$ replaced by $C^\infty(\Omega)$.   

It is expected that this law is somewhat universal in that it would hold
for a variety of subsets of $H$, the natural phase space for the NSE (see \eqref{ee31}).
We explore a particular
family of function classes $\calC(\sigma) \subset C^\infty(\Omega)$
for which all functions $u$ satisfy $\sup_{\alpha \in \bN}|A^{\alpha/2} u| \exp(-\sigma\alpha^2/2) < \infty$. Indeed, we show in Section \ref{sec:oneall} that the ``all for one, one for all" law holds for 
 $\bigcup_{\sigma>0}\calC(\sigma)$.  These classes are shown to be truly hierarchal: $\calC(\sigma_1) \subsetneqq \calC(\sigma_2)$ and moreover 
 $\bigcup_{\sigma>0}\calC(\sigma) \subsetneqq C^\infty(\Omega)$.

The distinction of the zero element in $H$ for a non-zero force $g$ is intriguing.
It is clear from the discussion in the opening paragraph that if $g\notin \D(A)$,
then $0 \notin \calA$.   In addition, we have a lower bound on $|u|$ for
$u \in \calA$ in the case $g \in \D(A^\alpha)$
which is valid for forces heavily weighted in the higher Fourier modes (see Theorem 12.2 in
\cite{DFJ05}).  The higher modes in the force must be more heavily weighted as the 
Grashof number $G$ (see \eqref{grashof}) is increased.  Due to its connection to the dissipation length scale, $G$ must be large for a 2D flow to be turbulent \cite{FMT93}.  It is unknown whether there exists any 
nonzero force for which $0 \in \calA$. In this paper we find a particular value 
$\sigma^*$, such that if $0 \in \calA$, then $g\in \calC(\sigma^*)$.  This narrows somewhat
the search for such a special force.

\section{Main results} 
We consider the Navier-Stokes equations (NSEs) on $\Omega=[0, L]^2$
\begin{align*}
\frac{\partial u}{\partial t}-\nu \Delta u + (u \cdot \nabla) u + \nabla p &= F,
\\
\nabla \cdot u &=0, \\
u(x,0) &=u_{0}(x),\\
\int_{\Omega} u dx=0,\qquad \int_{\Omega} F dx &=0,
\end{align*}
where $u:\mathbb{R}^2\rightarrow \mathbb{R}^2$ and $p:\mathbb{R}^2 \rightarrow \mathbb{R}$ are unknown, $\Omega$ -periodic functions, and $\nu>0$ is the kinematic viscosity of the fluid, $L>0$ is the period, $p$ is the pressure, and $F$ is the ``body'' force as in \cite{T95} \cite{CF89} \cite{T97}. We introduce the phase space $H$ as the subspace of $L^2({\Omega})^2$ consisting of the closure of the set of all $\mathbb{R}^2$-valued trigonometric polynomials $v$ such that
\begin{equation}
\label{ee31}
\nabla \cdot v=0 \quad \text{ and }  \quad \int_{\Omega}v(x)dx=0.
\end{equation}
The scalar product in $H$ is taken to be
\begin{equation*}
(u,v)=\int_{\Omega}u(x)\cdot v(x)dx
\end{equation*}
with associated norm $|u|=(u,u)^{\frac{1}{2}}$.

Let $\mathcal{P}:L^2(\Omega)^2 \rightarrow L^2(\Omega)^2$ be the orthogonal projection (called the Helmholtz-Leray projection) with range $H$, and define the Stokes operator as $A=-\mathcal{P}\Delta$ ($=-\Delta$, under periodic boundary conditions), which is positive, self-adjoint with a compact inverse.   As a consequence,  the space $H$ has an orthonormal basis $\{w_j\}_{j=1}^\infty$ of eigenfunctions of $A$, namely,  $A w_j=\lambda_j w_j$, with $0<\lambda _1=\left(2\pi/ L\right)^2
\leq \lambda _2\leq\lambda_3 \le \cdots$ \; (cf. \cite{CF89}, \cite{T97}).  The powers $A^\sigma$ are defined by
\begin{align}\label{powerdef}
A^\sigma v = \sum_{j=1}^\infty \lambda_j^\sigma (v,w_j) w_j\;, \quad \sigma \in \bR\;,
\end{align}
where $(\cdot,\cdot)$ is the $L^2-$scalar product.
The domain of $A^{\sigma}$ is denoted $\mathcal{D}(A^{\sigma})$.

The NSEs can be written as a differential equation (which will be referred to as the NSE) in the real Hilbert space $H$ in the following form
\begin{equation}
 \label{HNSE}
\frac{du}{d t}+\nu Au+B(u,u)=g, \quad u \in H,
\end{equation}
where the bilinear operator $B$ and the driving force $g$ are defined as \begin{equation*}
B(u,v)=\mathcal{P}((u\cdot\nabla)v)  \text{  and  }  g=\mathcal{P}F.
 \end{equation*} 

We recall that the global attractor $\mathcal{A}$ of the NSE is the collection of all elements $u_0$ in $H$ for which there exists a solution $u(t)$ of NSE, for all $t\in\mathbb{R}$, such that $u(0)=u_0$ and $\sup_{t\in \mathbb{R}} |u(t)|<\infty$.

To give another definition of $\mathcal{A}$, we need to recall several concepts. First, as is well-known, for any $u_0, f \in H$, there exists a unique continuous function $u$ from $[0,\infty)$ to $H$ such that $u(0)=u_0$, $u(t)\in \mathcal{D}(A)$, $t\in (0, \infty)$, and $u$ satisfies the NSE for all $t\in (0,\infty)$. Therefore, one can define the map $S(t): H\rightarrow H$ by
\begin{equation*}
S(t)u_0=u(t)
\end{equation*}
where $u(\cdot)$ is as above. Since $S(t_1)S(t_2)=S(t_1+t_2)$, the family $\big\{ S(t)\big\}_{t\geq 0}$ is called the ``solution" semigroup. Furthermore, a compact set $\mathcal{B}$ is called absorbing if for any bounded set $\tilde{\mathcal{B}}\subset{H}$ there is a time $\tilde{t}\geq 0$ such that $S(t)\tilde{\mathcal{B}}\subset \mathcal{B}$ for all $t\geq \tilde{t}$. The attractor can be now defined by the formula 
\begin{equation*}
 \mathcal{A}=\bigcap_{t\geq 0} S(t)\mathcal{B}, \label{attractor}
\end{equation*}
where $\mathcal{B}$ is any absorbing compact subset of $H$.

Let $\hc$ be the complex Hilbert space $H\otimes \mathbb{C} = H + iH$ (see discussion following 
\eqref{defcom} for more details). Similarly, for any linear subspace $D$ of $H$ we denote $D\otimes \mathbb{C}$ by $D_\mathbb{C}$.  For $\delta>0$ we define the strip 
	\begin{equation}
	\label{e21}
		\mathcal{S}(\delta):=\{\zeta\in\mathbb{C}:|\Im(\zeta)|<\delta\}.
	\end{equation}

\begin{thm}
\label{t2}
	If $0\in\mathcal{A}$, then there exists $\delta>0$ and $\Rt{\alpha}\in [0, \infty), \alpha\in \mathbb{N}$, such that  for any solution $u(\cdot)$ in $\mathcal{A}$, the function $A^{\frac{\alpha}{2}}u(\zeta)$ is $\hc$-valued analytic in the strip $\mathcal{S}(\delta)$, and $|A^{\frac{\alpha}{2}}u(\zeta)|\leq \Rt{\alpha}\nu\ko^\alpha$, where $u(\zeta)$ satisfies the NSE with complexified time \eqref{comnseq}.
\end{thm}

The proof of Theorem \ref{t2} will provide specific estimates for $\delta$ and $\Rt{\alpha}$, (see Remark \ref{r76}).

\begin{cor}
\label{t1}
	If $0\in\mathcal{A}$, then $\mathcal{A}\bigcup\{g\}\subset C^\infty([0,L]^2)$.
\end{cor}
The next result does not assume $0 \in \calA$, but rather, that $g$ is smooth to a certain extent.
\begin{prop}
\label{p23}
	Assume that $g\in\D(A^{\frac{\alpha-1}{2}})$ for some $\alpha\in\mathbb{N}$. Then $\mathcal{A}\subset\D(A^{\frac{\alpha+1}{2}})$ and any solution $u(\cdot)$ in $\mathcal{A}$ can be extended in the strip $\mathcal{S}(\delta_\alpha)$,	
where $\delta_{\alpha}>0$ depends on $g$ and $\alpha$, to a $\D(A^{\frac{\alpha+1}{2}})_\mathbb{C}-$ valued analytic function such that
	\begin{equation*}
		\sup\{|A^\frac{\alpha}{2} u(\zeta)|: \zeta\in \mathcal{S}(\delta_{\alpha})\} \leq m_{\alpha}\nu\ko^{\alpha}
	\end{equation*}
	where $m_{\alpha}$ is a non-dimensional parameter which, along with $\delta_\alpha$, depends only on $g$ and $\alpha$.
\end{prop}
The utility of the above proposition lies in the explicit estimates for the coefficients $\delta_\alpha$ and $m_{\alpha}$. Its proof is given in Section \ref{proof23sec}.

In fact, by using the techniques in the proof of Theorem \ref{t2} and Proposition \ref{p23} we will obtain the following generalization of Theorem \ref{t2}
\begin{thm}
\label{t24}
	If $\mathcal{A}\cap C^\infty([0,L]^2) \neq \emptyset$, then $ \mathcal{A}\bigcup\{g\}\subset C^\infty([0,L]^2)$.
\end{thm}

\section{Preliminary material}
Under periodic boundary conditions, we may express an element $u \in H $ as a Fourier series expansion
\begin{equation*}
u(x)=\sum_{k \in \mathbb{Z}^2\setminus \{0\}}\hat{u}(k)e^{i\kappa_{0}k\cdot x},
\end{equation*}
where
$\kappa_{0}
=2\pi/L$, $\hat{u}(0)=0$, $(\hat{u}(k))^{\ast}=\hat{u}(-k)$ and $k\cdot\hat{u}(k)=0.$
Parseval's identity reads as
\begin{equation*}
|u|^2=L^2\sum_{k\in\mathbb{Z}^2\setminus \{0\}}\hat{u}(k)\cdot \hat{u}(-k)=L^2\sum_{k\in\mathbb{Z}^2\setminus \{0\}}|\hat{u}(k)|^2,
\end{equation*}
as well as
\begin{equation*}
(u,v)=L^2\sum_{k\in\mathbb{Z}^2\setminus \{0\}}\hat{u}(k)\cdot \hat{v}(-k).
\end{equation*}

The following inequalities will be repeatedly used in this paper
\begin{align}
\kappa_0|u| &\leq |A^{\frac{1}{2}}u|, \quad\mbox{for $u\in\D(A^{\frac{1}{2}})$,}\label{poincare} \\
|u|_{L^4(\Omega)}&\leq c_{L}|u|^{\frac{1}{2}}|A^{\frac{1}{2}}u|^{\frac{1}{2}}, \quad\mbox{for $u\in\D(A^{\frac{1}{2}})$,}\label{ladyzh}\\
|u|_{\infty}&\leq c_{A}|u|^{\frac{1}{2}}|Au|^{\frac{1}{2}}, \quad\mbox{for $u\in\D(A)$}. \label{agmon}
\end{align}
known respectively as the Poincar\'{e}, Ladyzhenskaya  and Agmon inequalities. Both $c_L$ and $c_A$ are absolute constants. By Theorem 9.2 and 9.3 in \cite{FJY12}, we have that
\begin{align*}
c_L&\leq \big(\frac{1}{(2\pi)^2}+\frac{1}{\sqrt{2}\pi}+2\big)^{\frac{1}{4}},\\
c_A&\leq \big( \frac{1}{(2\pi)^2}+\frac{1}{\sqrt{2}\pi}+2+4\sqrt{2}\big)^{\frac{1}{2}}.
\end{align*}
In fact, (\ref{agmon}) can be given in the following stronger form

We stress that our estimates will depend on the generalized Grashof number
\begin{equation}
\label{grashof}
G=\frac{|g|}{{\nu}^2 \kappa_0^2}. 
\end{equation}

We also recall that if $G<c_L^{-2}$ then $\mathcal{A}$ contains only one point $u_0\in\D(A)$ which satisfies 
\begin{equation*}
	\nu Au_0 + B(u_0, u_0) = g
\end{equation*}
(see Proposition 2.1 in \cite{DFJ05}). Note that in this case both Theorem \ref{t1} and \ref{t2} are trivially valid. Therefore, throughout this paper we will assume that $G$ satisfies 
\begin{equation}
\label{e314}
	G\geq \frac{1}{c_L^2}.
\end{equation}


We recall the following algebraic  properties of the bilinear operator $B(u,v)$ where $u, v,w\in \mathcal{D}(A)$ from \cite{DFJ05}
\begin{align}
& (B(u,v),w)=-(B(u,w),v), \label{biop1} \\
& (B(u,u),Au)=0, \label{biop2} \\
& (B(Av,v),u)=(B(u,v),Av), \label{biop3} \\
& (B(u,v),Av)+(B(v,u),Av)+(B(v,v),Au)=0. \label{biop4}
\end{align}
From (\ref{biop3}) and (\ref{biop4}), it easily follows that if $u\in\mathcal{D}(A^{3/2})$ then $B(u,u)\in \mathcal D(A)$ and
\begin{equation} \label{abiop}
AB(u,u)=B(u,Au)-B(Au,u).
\end{equation}

If we multiply (\ref{HNSE}) by $u$ and $Au$, respectively, integrate over $\Omega$, and apply the relations (\ref{biop1}) and (\ref{biop2}), then we have the following inequalities
\begin{equation}
\frac{1}{2} \frac{d}{dt} |u|^2+\nu\ko^2 |u|^2\leq\frac{1}{2} \frac{d}{dt} |u|^2+\nu |A^{\frac{1}{2}}u|^2=(g,u)\leq \frac{|g|^2}{2\nu\ko^2}+\frac{\nu\ko^2}{2}|u|^2, \label{energyeq}
\end{equation}
\begin{equation}
\frac{1}{2} \frac{d}{dt} |A^{\frac{1}{2}}u|^2+\nu |Au|^2=(g,Au)\leq\frac{|g|^2}{2\nu}+\frac{\nu|Au|^2}{2}. \label{enstrophyeq}
\end{equation}
(\ref{energyeq}) and (\ref{enstrophyeq}) are called the balance equations for the energy and enstrophy, respectively. Applying Gronwall's lemma to (\ref{energyeq}) and (\ref{enstrophyeq}) we obtain, for all $t\geq t_0$, that   
\begin{equation}
|u(t)|^2\leq e^{-\nu\ko^2(t-t_0)}|u(t_0)|^2+(1-e^{-\nu\ko^2(t-t_0)}) {\nu}^2G^2, \label{energyieq}
\end{equation}
\begin{equation}
|A^{\frac{1}{2}}u(t)|^2\leq e^{-\nu\ko^2(t-t_0)}|A^{\frac{1}{2}}u(t_0)|^2+(1-e^{-\nu\ko^2(t-t_0)}) {\nu}^2\ko^2G^2. \label{enstrophyieq}
\end{equation}
From (\ref{enstrophyieq}) we see that the closed ball $B$ of radius $2\nu\kappa_0 G$ in $\D(A^{1/2})$, by the Sobolev embedding theorem, is an absorbing set in $H$. Therefore, we can define the global attractor $\mathcal{A}$ as in  (\ref{attractor}).

In the next lemma, we list several necessary estimates involving $B(\cdot,\cdot)$.

\begin{lem} 
\label{bioplist}
The following hold in the appropriate space,
\begin{align}
\label{bieq2} |(B(u,u),A^{2}u)|&\leq 2c_L^2|Au||A^{\frac{3}{2}}u||A^{\frac{1}{2}}u|,\, u\in\D(A^2),\\
\label{bieq3}|(B(u,u),A^{3}u)|&\leq \sqrt{2}(\sqrt{2}c_L^2+c_A)|u|^{\frac{1}{2}}|Au|^{\frac{1}{2}}|A^{\frac{3}{2}}u||A^2 u|,\, u\in\D(A^3) .
\end{align}
\end{lem}
Relations (\ref{bieq2}) and (\ref{bieq3}) are straightforward applications of the Ladyzhenskaya and Agmon's inequalities.

Now we consider the NSE with complexified time and the corresponding solutions in $H_{\mathbb{C}}$ as in \cite{FHN07} and \cite{FJK96}. We recall that
\begin{equation} 
\label{defcom}
H_{\mathbb{C}}=\{u+iv: u,v\in H\},
\end{equation}
and that $H_{\mathbb{C}}$ is a Hilbert space with respect to the following inner product
\begin{equation*} 
\label{cominn}
(u+iv,u'+iv')_{H_{\mathbb{C}}}=(u,u')_H+(v,v')_H+i[(v,u')_H-(u,v')_H],
\end{equation*}
where $u$, $u'$, $v$, $v'\in H$. The extension $A_{\mathbb{C}}$ of $A$ is given by
\begin{equation*}
A_{\mathbb{C}}(u+iv)=Au+iAv,
\end{equation*}
for $u$, $v\in \mathcal{D}(A)$; thus $\D(\ac) = \D(A)_\mathbb{C}$. Similarly, $B(\cdot,\cdot)$ can be extended to a bounded bilinear operator from $\D(\ac^{\frac{1}{2}})\times\D(\ac)$ to $\hc$ by the formula
\begin{equation*}
B_{\mathbb{C}}(u+iv,u'+iv')=B(u,u')-B(v,v')+i[B(u,v')+B(v,u')],
\end{equation*}
for $u$, $v\in \mathcal{D}(A^{\frac{1}{2}})$, $u'$, $v'\in \mathcal{D}(A)$. We should note that (\ref{biop1})-(\ref{abiop}) do not hold in the complex case.

The Navier-Stokes equations with complex time are defined as
\begin{align}
\label{comnseq}
\frac{d u(\zeta)}{d \zeta}+\nu A_{\mathbb{C}} u(\zeta)+B_{\mathbb{C}}(u(\zeta),u(\zeta))=g, 
\end{align}
where $\zeta\in \mathcal{S}(\delta)$ (see (\ref{e21})), $u(\zeta)\in H_{\mathbb{C}}$, and $\frac{du(\zeta)}{d\zeta}$ denotes the derivative of $H_{\mathbb{C}}$-valued analytic functions $u(\zeta)$.

For notational simplicity, in the following considerations we drop the subscript $\mathbb{C}$ for the inner products, norms, and the operators just defined. 	
\section{Supplementary estimates}
We now obtain estimates for the nonlinear terms with complexified time, observing that neither relation (\ref{biop2}) nor Lemma \ref{bioplist} hold in this case. We will use the Ladyzhenskaya and Agmon inequalities as described before, noting the additional factor 2.
\begin{eqnarray} 
\label{lady} 
|u|_{L^4} \leq 2c_L |u|^{\frac{1}{2}} |A^{\frac{1}{2}} u|^{\frac{1}{2}},\hspace {0.3 in} \mbox{where} \hspace{0.1 in} u \in \mathcal{D}(A^{\frac{1}{2}})_{\mathbb C} 
\end{eqnarray} 
and
\begin{eqnarray} 
\label{Agmon} 
|u|_{\infty} \leq 2c_A |u|^{\frac{1}{2}} |Au|^{\frac{1}{2}}, \hspace {0.3 in} \mbox{where} \hspace{0.1 in} u \in \mathcal{D}(A)_{\mathbb C}.
\end{eqnarray}

In the present complex case, the analogue of Lemma \ref{bioplist} is
\begin{lem}
\label{lemma51}
 For $u \in \mathcal{D}(A^{\alpha})_\mathbb{C}$, where $\alpha$ = $1$, $2$ or $3$, 
 \begin{align}
   |(B(u,u),Au)| \leq 4 c^2_L |u|^{\frac{1}{2}} |A^{\frac{1}{2}}u| |Au|^{\frac{3}{2}},\\
  \label{alpha4}|(B(u,u),A^2u)| \leq 2(2 c^2_L + c_A) |u|^{\frac{1}{2}} |Au|^{\frac{3}{2}} |A^{\frac{3}{2}}u|,\\
  \label{alpha6}|(B(u,u),A^3u)| \leq 2(2 c^2_L+ c_A) |u|^{\frac{1}{2}} |Au|^{\frac{3}{2}}  |A^{\frac{5}{2}}u|.  
  \end{align}
  \end{lem}
  \begin{proof}
      To obtain the first inequality, use (\ref{lady}) for the first two terms and the $L^2$ norm for the third term.
   For the second inequality, we use integration by parts to get 
   \begin{align} 
   \label{sup1}  
   |B(u,u), A^2u)| \leq \displaystyle\sum_{j=1,2}[|(B(D_ju, u),D_jAu)| +|(B(u, D_ju),D_jAu)|].  
   \end{align} 
   
  Using (\ref{lady}) and (\ref{Agmon}), we have 
   \begin{eqnarray} 
   \label{sup2}
   \sum_j|(B(D_ju, u),D_jAu)| \leq 4 c^2_L|A^{\frac{1}{2}}u| |Au| |A^{\frac{3}{2}}u| \hspace{ 0.2in }
   \end{eqnarray} 
   and 
   \begin{eqnarray} 
   \label{sup3}  
   \sum_j|(B(u, D_ju),D_jAu)| \leq 2c_A|u|^{\frac{1}{2}} |Au|^{\frac{3}{2}} |A^{\frac{3}{2}}u|,
   \end{eqnarray}
   for $u \in \mathcal{D}(A^{\frac{3}{2}})$. 
   
   Now apply the interpolating relation $|A^{\frac{1}{2}}u| \leq |Au|^{\frac{1}{2}} |u|^{\frac{1}{2}}$ in (\ref{sup2}) to arrive at (\ref{alpha4}).
   
   For the last inequality, we use the same method as above. Integrating by parts, using (\ref{lady}) and (\ref{Agmon}),  and then interpolating, we have
     \begin{align*}
      \label{sup4}  
      |B(u,u), A^3u)| &\leq \displaystyle\sum_{j=1,2}[|(B(D_ju, u),D_jA^2u)| +|(B(u, D_ju),D_jA^2u)|]\\
    &\leq   4 c^2_L |A^{\frac{1}{2}}u| |Au| |A^{\frac{5}{2}}u| + 2 c_A |u|^{\frac{1}{2}} |Au|^{\frac{3}{2}}  |A^{\frac{5}{2}}u|\nonumber\\ 
    &\leq 2(2 c^2_L+ c_A) |u|^{\frac{1}{2}} |Au|^{\frac{3}{2}}  |A^{\frac{5}{2}}u|, \nonumber
    \end{align*}
thus obtaining (\ref{alpha6}).    
    
   \end{proof}
Concerning the existence of $\mathcal{D}(A^{\frac{1}{2}})_\mathbb{C}$-valued analytic extensions of the solutions of the NSE, one can consult Section 7 in Part I of \cite{T95} and Chapter 12 in \cite{CF89}. However, for our work, we need the following observations.
\begin{remark}
	Many of the differential relations to follow are of the form
	\begin{equation*}
		\frac{d}{d\rho}\Phi(u(t_0+\rho e^{i\theta})) = \Psi(u(t_0+\rho e^{i\theta})),
	\end{equation*}
	where $\Phi(u)$ and $\Psi(u)$ are explicit functions of $u$ in a specified subspace of $H$.
	Often the definition of $\Psi(u)$ involves many terms. Therefore in the sequel we will make the following abuse of notation
		\begin{equation*}
		\frac{d}{d\rho}\Phi(u(t_0+\rho e^{i\theta})) = \Psi(u).
	\end{equation*}
\end{remark}
\begin{remark}
\label{r52}
Let $\alpha \in \mathbb{N}$. To solve the equation (\ref{comnseq})
in a strip $\mathcal{S}(\delta_{\alpha})$ and to insure that $u(\zeta)$ is a $\mathcal{D}(A^{\alpha/2})_{\mathbb{C}}$-valued analytic function (equivalently, $A^{\alpha/2}u(\zeta)$ is $H_{\mathbb{C}}$-valued analytic), the proof for the case $\alpha=1$ presented in \cite{T95} and \cite{CF89} shows that it suffices to establish the following fact:

For any $t_0\in \mathbb{R}$, $\theta\in [-\frac{\pi}{4}, \frac{\pi}{4}]$ and solution $u(\zeta)$ of the equation (\ref{comnseq}) in $\mathcal{S}(\delta_{\alpha})$, the solution of the equation
\begin{align}
\label{annse}
\frac{d}{d\rho} u(t_0+\rho e^{i \theta})&+\nu (\cos \theta) Au+B(u, u)=g, \quad g\in\D(A^{\frac{\alpha-1}{2}})
\end{align}
satisfies, for
\begin{equation*}
0\leq \rho \leq \frac{\delta_\alpha}{\sin{\pi/4}} = \sqrt{2}{\delta_{\alpha}}, 
\end{equation*}
the following conditions
\begin{equation*}
u(t_0+\rho e^{i\theta})\in \mathcal{D}(A^{\frac{\alpha+1}{2}})_{\mathbb{C}},
\end{equation*}
and $\sup |A^{\frac{\alpha+1}{2}}u(t_0+\rho e^{i\theta})|$ is finite and independent of $t_0$, $\rho$ and $\theta$.

This can be rigorously established with ``the Galerkin approximation, for which analyticity in time is trivial because it is a finite-dimensional system with a polynomial nonlinearity. The crucial part, then, is to obtain suitable a priori estimates for the solution in a complex time region that is independent of the Galerkin approximation.''(see \cite{F01} Chapter II, Section 8, Page 63).  The justification for this procedure is given in the Appendix.
\end{remark}

Using the procedure described in Remark \ref{r52}, we will prove Theorem \ref{t2} by induction on $\alpha$. In order to start a uniform recurrent process we  need $\alpha \geq 3$. In the following Lemmas \ref{lemma54}, \ref{lemma58} and \ref{lemma59} we obtain the necessary estimates for $\alpha = 1,2,3$. We stress that the case $\alpha = 1$ was treated in Theorem 12.1 in \cite{CF89}, while the cases $\alpha = 1,2$ were already established in \cite{DFJ05}, Theorem 11.1, although with different estimates.

\begin{lem}
\label{lemma54}
If $u(\cdot)$ is a solution of the NSE in the attractor $\mathcal{A}$, then 
\begin{enumerate}[(i)]
\item $u(\cdot)$ can be extended to a $\mathcal{D}(A^{\frac{1}{2}})_{\mathbb C}$ -valued analytic function in the strip $\mathcal{S}(\delta_1)$, where 
\begin{equation} 
\label{alpha11}
\delta_1 := \frac{1}{16\cdot 24^3 c_L^8 \nu \ko^2 G^4} \end{equation} and 
\begin{equation} 
\label{alpha12} 
|A^{\frac{1}{2}} u(\zeta)| \leq  \Rt{1}\nu \kappa_0 , \hspace{ 0.2 in } \forall\ \zeta \in \mathcal{S}(\delta_1),
\end{equation} 
where 
\begin{equation}
\label{R1t} 
\Rt{1} = \sqrt{2}G. 
\end{equation} 
\item Moreover,  
defining 
\begin{equation} 
\label{R2}
R_2 = 2137G^3c_L^4, 
\end{equation}  
we have
\begin{align} 
\label{alpha13} 
|Au(t)| \leq  \R{2}\nu\ko^2, \hspace{ 0.1 in} \forall\  t \in \mathbb R.
\end{align} 
\end{enumerate}
\end{lem}

\begin{proof}
First, according to Remark \ref{r52} to prove the statement (i) it is sufficient to establish the estimates (\ref{alpha12}) and (\ref{R1t}) for $\delta_1$ chosen as in (\ref{alpha11}).

Taking the inner product of both sides in (\ref{annse}) with $Au(t_{0}+\rho e^{i \theta})$, we obtain
\begin{align*} 
\frac{1}{2}\frac{d}{d\rho}|A^{\frac{1}{2}} u(t_{0}+\rho e^{i\theta})|^2  &= \Re(e^{i\theta} (g, Au)) - \nu\cos\theta |Au|^2- \Re(e^{i \theta}(B(u, u), Au)) \\
&\leq |g||Au| -   \frac{\nu}{\sqrt{2}} |Au|^2 + |(B(u, u), Au))|. \nonumber
\end{align*}
   Using the Cauchy-Schwarz inequality and then the Young's inequality $|ab| \leq |a|^p/p+|b|^q/q$ with $p=q=2$, for the term  $|g||Au|$, we obtain
 \begin{align*} 
 \frac{d}{d\rho}|A^{\frac{1}{2}} u(t_{0}+\rho e^{i\theta})|^2 +  \frac{\nu}{\sqrt{2}} |Au|^2  \leq \sqrt{2} \frac{|g|^2}{\nu}    + |2(B(u, u), Au)|. \end{align*}
We use Lemma \ref{lemma51} for the bilinear term in the above relation to get the following inequality

\begin{align*}   
\frac{d}{d\rho}|A^{\frac{1}{2}} u(t_{0}+\rho e^{i\theta})|^2 +  \frac{\nu}{\sqrt{2}} |Au|^2  \leq \sqrt{2}\frac{|g|^2}{\nu}    + 8 c^2_L |u|^{\frac{1}{2}} |A^{\frac{1}{2}}u| |Au|^{\frac{3}{2}}.
\end{align*}
Using Young's inequality again, with $p=4$ and $q=4/3 $ for the last term, we have 
\begin{align*}  
8 c^2_L |u|^{\frac{1}{2}} |A^{\frac{1}{2}}u| |Au|^{\frac{3}{2}} \leq& \frac{3}{4} \bigg(\nu^{\frac{3}{4}}\bigg(\frac{\sqrt{2}}{3}\bigg)^{\frac{3}{4}} |Au|^{\frac{3}{2}}\bigg)^{\frac{4}{3}} + \frac{1}{4} \bigg( \frac{1}{\nu^{\frac{3}{4}}} \bigg( \frac{3}{\sqrt{2}}\bigg)^{\frac{3}{4}} 8 c^2_L|u|^{\frac{1}{2}} |A^{\frac{1}{2}} u|\bigg)^4, 
\end{align*}
 and hence,
\begin{align} 
\label{maineq}
\frac{d}{d\rho}|A^{\frac{1}{2}} u(t_{0}+\rho e^{i\theta})|^2 +  \frac{\nu}{2\sqrt{2}} |Au|^2  \leq \sqrt{2}\frac{|g|^2}{\nu}    + \frac{8^3 27  c^8_L}{\nu^3 \sqrt{2}} |u|^2 |A^{\frac{1}{2}}u|^4. 
\end{align}
From (\ref{poincare}) and (\ref{maineq}), we obtain  
\begin{align*}  
\frac{d}{d\rho}|A^{\frac{1}{2}} u(t_{0}+\rho e^{i\theta})|^2 \leq \sqrt{2}\frac{|g|^2}{\nu}    + \frac{8^3 27  c^8_L}{\nu^3 \kappa^2_0 \sqrt{2}}  |A^{\frac{1}{2}}u|^6. 
\end{align*}
The above inequality has the form 
\begin{align} 
\label{alpha15} 
\frac{d\phi}{d\rho} \leq \gamma + \beta \phi^3,
\end{align} 
where 
\begin{equation*}
\phi(\rho) := |A^{\frac{1}{2}} u(t_0+\rho e^{i \theta})|^2,\quad \gamma =  \sqrt{2}\frac{|g|^2}{\nu},\quad \beta=  \frac{24^3  c^8_L}{\nu^3 \kappa^2_0 \sqrt{2}}.
\end{equation*}
Integrating (\ref{alpha15}), we obtain 
\begin{align*} 
\int_{\phi(0)}^{\phi(\rho)} \frac{d\phi}{(\gamma^{\frac{1}{3}} + \beta^{\frac{1}{3}} \phi)^3} \leq \int_{\phi(0)}^{\phi(\rho)} \frac{d\phi}{(\gamma + \beta \phi^3)} \leq \rho,
\end{align*}
and hence
\begin{align}
\label {alpha100} 
\frac{1}{2\beta^{\frac{1}{3}} (\gamma^{\frac{1}{3}}  + \beta^{\frac{1}{3}} \phi(0))^2} - \frac{1}{2\beta^{\frac{1}{3}} (\gamma^{\frac{1}{3}}  + \beta^{\frac{1}{3}} \phi(\rho))^2}  \leq \rho. 
\end{align} 
 Thus, if
\begin{align} 
\label{e529}
\rho \leq \frac{1}{ 4 \beta^{\frac{1}{3}} (\gamma^{\frac{1}{3}} + \beta^{\frac{1}{3}} \phi(0))^2}
\end{align}
then $\phi(\rho)$ satisfies 
\begin{align*} 
\gamma^{\frac{1}{3}} + \beta^{\frac{1}{3}} \phi(\rho) \leq \sqrt{2}(\gamma^{\frac{1}{3}} + \beta^{\frac{1}{3}} G^2 (\nu \kappa_0)^2), 
\end{align*}
that is
\begin{align}  
\label{alpha16} 
|A^{\frac{1}{2}}u(t_0 + \rho e^{i \theta})|^2 &\leq (\sqrt{2} -1) \bigg(\frac{\gamma}{\beta}\bigg)^{\frac{1}{3}} + \sqrt{2} G^2 (\nu \ko)^2 \\
&\leq \frac{2^{\frac{1}{3}}(|g| \nu \kappa_0)^{2/3}}{24 c^{8/3}_L} + \sqrt{2} G^2 (\nu \ko)^2 \nonumber\\
&\leq \bigg( \sqrt{2} + \frac{2^{\frac{1}{3}}}{24} \bigg) G^2 (\nu \ko)^2 \leq 2G^2(\nu \ko)^2.\nonumber
\end{align}

Note that in the third inequality above  we used (\ref{e314}).

If $\delta_1$ is defined as in (\ref{alpha11}) and if $\rho \leq \sqrt{2} \delta_1$, then (\ref{e529}) holds. Consequently, (\ref{alpha16}) also holds. That is
\begin{equation*}
|A^{\frac{1}{2}}u(t_0 + \rho e^{i \theta})|\leq\Rt{1}\nu\ko,
\end{equation*}
where $\Rt{1}$ is defined in (\ref{R1t}).

Since $\theta\in[-\frac{\pi}{4},\frac{\pi}{4}]$ and $t_0\in\mathbb{R}$ are arbitrary, we infer
\begin{equation*}
|A^{\frac{1}{2}}u(\zeta)|\leq\Rt{1}\nu\ko, \quad \mbox{for $\zeta\in \mathcal{S}(\delta_1)$}.
\end{equation*}

With this estimate, the proof of statement (i) is concluded.

It remains to prove statement (ii). Integrating (\ref{maineq}) and applying (\ref{e314}), we obtain
\begin{align*}
\frac{\nu}{2\sqrt{2}}\int_{0}^{\sqrt{2}\delta_1} |Au(t_0+\rho e^{i \pi/4})|^2 d\rho &\leq  |A^{\frac{1}{2}}u(t_0)|^2 + \int_0^{\sqrt{2}\delta_1} (\gamma + \beta |A^{\frac{1}{2}}u(t_0+\rho e^{i \pi/4})|^6)d\rho,\nonumber\\
&\leq  2G^2 \nu^2 \ko^2 +\sqrt{2}\gamma \delta_1 + 8\sqrt{2}G^6\nu^6\ko^6\beta\delta_1\nonumber\\
&=\left[2+\frac{1}{8\cdot 24^3c_L^8G^4}+\frac{1}{2}\right]G^2\nu^2\ko^2
\nonumber\\
&\leq \left[2+\frac{1}{8\cdot 24^3}+\frac{1}{2}\right]G^2\nu^2\ko^2 \leq 2\sqrt{2}G^2\nu^2\ko^2,
\end{align*} 
i.e.,
\begin{align}
\label{M}
\int_{0}^{\sqrt{2}\delta_1} |Au(t_0+\rho e^{i \pi/4})|^2 d\rho &\leq 8G^2\nu\ko^2.
\end{align}
Since $Au(\zeta)$ is an analytic function in $D(t_0,\delta):=\{s_1+is_2: |s_1-t_0|^2 + s^2_2 \leq \delta_1^2\}$, it satisfies the mean value property
\begin{equation*} 
\label{e534}
Au(t_0) = \frac{1}{\pi \delta^2_1} \iint_{D(t_0,\delta_1)} Au(s_1+is_2) ds_1ds_2, 
\end{equation*} 
from which we deduce 
\begin{align*} 
|Au(t_0)| \leq \frac{1}{\pi \delta^2_1} \iint_{D(t_0,\delta_1)} |Au(s_1+is_2)| ds_1 ds_2. 
\end{align*} 
In order to exploit the estimate (\ref{M}), we replace the disk $D(t_0, \delta_1)$ by the polygon $abcdef$ as shown in the figure below.
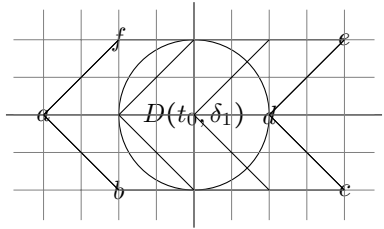
\begin{center}
    \begin{tikzpicture}
        \draw (-2,0)node{$a$} -- (-1,-1)node{$b$} -- (2, -1)node{$c$} -- (1, 0)node{$d$} -- (2, 1)node{$e$} -- (-1,1)node{$f$} --(-2,0);
        \draw (-2.5,0) -- (2.5,0);
        \draw (0,-1.5) -- (0,1.5);
        \draw (0,0) node{$D(t_0,\delta_1)$} circle(1cm);
        \draw [step=0.5cm, gray, very thin] (-2.4,-1.4) grid (2.4, 1.4);
        \foreach \x in {-2,-1,...,1}
            {
                \draw (\x,0)--(\x+1,1);
                \draw (\x,0)--(\x+1,-1);
            }
    \end{tikzpicture}
\end{center}

Then, by using Schwarz reflection principle, we obtain
\begin{align*} 
\label{e536}
|Au(t_0)| &\leq \frac{1}{\pi \delta_1^2} \int_{abcdef} |Au(s_1+is_2)| ds_1ds_2 \\
&= \frac{2}{\sqrt{2} \pi \delta_1^2} \displaystyle\int_{t_0-2\delta_1}^{t_0 + \delta_1} \displaystyle\int_{0}^{\sqrt{2}\delta_1} |Au(t+\rho e^{i\pi/4})| d\rho dt \nonumber\\
&\leq  \frac{2}{\sqrt{2} \pi \delta_1^2} \displaystyle\int_{t_0-2\delta_1}^{t_0 + \delta_1}  \bigg(\displaystyle\int_{0}^{\sqrt{2}\delta_1} |Au(t+\rho e^{i\pi/4})|^2 d\rho \bigg)^{\frac{1}{2}} (\sqrt{2}\delta_1)^{\frac{1}{2}}dt\nonumber\\
&\leq \frac{12\cdot 2^{\frac{1}{4}} }{\pi} \bigg(\frac{G^2\nu\ko^2}{\delta_1}\bigg)^{\frac{1}{2}},\nonumber
\end{align*}
that is   
\begin{equation*} 
\label{areal} 
|Au(t_{0})| \leq \frac{6\cdot 2^8 \cdot 3\sqrt{3}}{2^{\frac{1}{4}}\pi}G^3c_L^4\nu\ko^2 \leq \R{2}\nu\ko^2.
\end{equation*}   
This completes the proof of the statement (ii) and Lemma \ref{lemma54}.
\end{proof}
\begin{cor}
  For all $u^0\in \mathcal{A}$, we have
\begin{equation}
\label{e519}
	|A^{\frac{1}{2}}u^0|\leq \R{1}\nu\ko,\quad \R{1} := G,
\end{equation}
and
\begin{equation}
\label{e520}
	|Au^0|\leq \R{2}\nu\ko^2.
\end{equation}
\end{cor}
\begin{proof}
Let $u^0\in \mathcal{A}$ and denote by $u(t)$, $t\in\mathbb{R}$, the solution of the NSE satisfying $u(0) = u^0$. Then, according to Lemma \ref{lemma54}, (\ref{alpha13})  
holds; in particular, for $t=0$. This yields (\ref{e520}). The estimate (\ref{e519}) follows from (\ref{enstrophyieq}) with $t=0$ and $t_0\rightarrow -\infty$.
\end{proof}

\begin{remark}
	It is worth comparing (\ref{alpha13}) with 
	\begin{equation}
	\label{e-540}
		|Au|^2\leq \nu^2\ko^4G^2(2\Lambda_1^{\frac{1}{2}}+c_L^2G^2)
\end{equation}		
from Theorem 3.1 of \cite{DFJ05}, where $\Lambda_1 := \frac{|A^{\frac{1}{2}}g|^2}{\ko^2|g|^2}$. When $\Lambda_1^{\frac{1}{2}}> \frac{2137^2}{2}c_L^8G^4$ the estimate (\ref{e520}) is better than (\ref{e-540}). Moreover, there are cases when $\Lambda_1$ is large, $\mathcal{A}=\{u^0\}$, and thus $|Au^0|\leq G\nu\ko^2$. 
\end{remark}
\begin{cor}
\label{cor51} 
If $0\in\mathcal{A}$ then $g\in\D(A^{\frac{1}{2}})$, and 
\begin{equation}
\label{e540}
|A^{\frac{1}{2}} g| \leq \frac{\Rt{1} \nu \kappa_0}{\delta_1}, 
\end{equation}
where $\delta_1$ and $\Rt{1}$ are defined as in (\ref{alpha11}) and (\ref{R1t}) respectively.
\end{cor}
\begin{proof}
Let $u(t)$, $t\in \mathbb{R}$, be a solution of NSE such that $u(0) = 0$. According to Theorem 11.1 in \cite{DFJ05}, if $0\in\mathcal{A}$, then $g\in \D(A)$. We  evaluate the NSE at $t_0 =0$ and $\theta=0$ to obtain
\begin{equation*}
\frac{du(\zeta)}{d\zeta}\bigg|_{t_0=0} = g. 
\end{equation*}
Since $u(\zeta)$ is a $\D(A^{\frac{1}{2}})_{\mathbb{C}}$-valued analytic function, its derivative $\frac{du(\zeta)}{d\zeta}\in\D(A^{\frac{1}{2}})_{\mathbb{C}}$ for all $\zeta\in \mathcal{S}(\delta_1)$.
Thus, $g\in\D(A^{\frac{1}{2}})$, and then from
\begin{align*} 
A^{\frac{1}{2}}g &= A^{\frac{1}{2}}\frac{du(\zeta)}{d\zeta}|_{\zeta=0}= \frac{dA^{\frac{1}{2}}u(\zeta)}{d\zeta}|_{\zeta=0}=\frac{1}{2 \pi i} \int_{\partial D(0,\delta)} \frac{A^{\frac{1}{2}}u(z)}{z^2} dz,
\end{align*}
where $\delta\in(0,\delta_1)$, we obtain 
\begin{equation}
\label{e5422} 
|A^{\frac{1}{2}} g| \leq \frac{\Rt{1}\nu \kappa_0}{\delta}. 
\end{equation}
Letting $\delta\rightarrow\delta_1$ in (\ref{e5422}), we deduce (\ref{e540})
\end{proof}

We will now establish estimates for the case $\alpha = 2$.
\begin{lem}
\label{lemma58}
If $0\in\mathcal{A}$ and if $u(t)$, $t\in \mathbb{R}$ is any solution of the NSE in $\mathcal{A}$, then $u(t)$ can be extended to a $\mathcal{D}(A)_{\mathbb C}$-valued analytic function $u(\zeta)$, for $\zeta \in \mathcal{S}(\delta_2)$, where 
\begin{align}
\label{delta2}
\delta_2 &:= \min\left\{\delta_1, 16^{-1}\left[ (2c_L^2+c_A)^{\frac{8}{3}}\Rt{1}^{\frac{8}{3}}(\frac{\nu\ko^2}{8\delta_1^2})^{\frac{2}{3}}+(2c_L^2+c_A)^4\Rt{1}^2\R{2}^2(\nu\ko^2)^2\right]^{-\frac{1}{2}}\right\}, 
\end{align}
and $\delta_1$, $\Rt{1}$ and $\R{2}$ are defined in (\ref{alpha11}), (\ref{R1t}) and (\ref{R2}), respectively.
Furthermore, 
\begin{equation}
\label{e542}
|Au(\zeta)| \leq \Rt{2}\nu\ko^2,  \quad{\mbox{
for $\zeta \in \mathcal{S}(\delta_2)$, }}
\end{equation}
where 
\begin{equation}
\label{rt2}
	\Rt{2} := \left(\frac{3(\sqrt{2}\cdot 16^2\cdot 24^6c_L^{16})^{2/3}}{4(2c_L^2+c_A)^{4/3}}G^6+4\R{2}^2\right)^{\frac{1}{2}}.
\end{equation}
\end{lem}

\begin{proof}
Applying again the short procedure in Remark \ref{r52}, we take the inner product of the NSE with the function $A^2u(t_0 +\rho e^{i \theta})$ and obtain
\begin{align*}  
\frac{1}{2}\frac{d}{d\rho}|Au(t_0 + \rho e^{i \theta})|^2 = \Re(e^{i\theta}(g, A^2u) ) &- \nu(\cos\theta) |A^{\frac{3}{2}} u|^2 - \Re(e^{i \theta}(B(u,u),A^2u)).  
\end{align*}
Proceeding as in the proof of Lemma \ref{lemma54}, we obtain
\begin{align*} 
 \frac{1}{2}\frac{d}{d\rho}|Au(t_0 + \rho e^{i \theta})|^2 &\leq  |(A^{\frac{1}{2}}g, A^{\frac{3}{2}}u)| -\frac{\nu}{\sqrt{2}} |A^{\frac{3}{2}} u|^2 +  |(B(u,u),A^2u)|.\nonumber 
 \end{align*}
 Using Young's inequality $ |ab| \leq \frac{1}{p} |a|^p + \frac{1}{q} |b|^q$, with $p=q=2$ for the term $|(A^{\frac{1}{2}}g, A^{\frac{3}{2}}u)|$, we obtain
\begin{align*}
\frac{d}{d\rho}|Au(t_0 + \rho e^{i \theta})|^2 &\leq  \sqrt{2}\frac{|A^{\frac{1}{2}}g|^2}{\nu} -\frac{\nu}{\sqrt{2}} |A^{\frac{3}{2}} u|^2+ 2|(B(u,u),A^2u)|.  \nonumber
 \end{align*} 
We use Lemma \ref{lemma51} to obtain
\begin{align*} 
\frac{d}{d\rho}|Au(t_0 + \rho e^{i \theta})|^2 &\leq \sqrt{2}\frac{|A^{\frac{1}{2}}g|^2}{\nu}-\frac{\nu}{\sqrt{2}} |A^{\frac{3}{2}} u|^2+  4(2 c^2_L + c_A) |u|^{\frac{1}{2}} |Au|^{\frac{3}{2}} |A^{\frac{3}{2}}u| . 
\end{align*}
Using Young's inequality again, we obtain
\begin{align*} 
\frac{d}{d\rho}|Au(t_0 + \rho e^{i \theta})|^2  &+ \frac{1}{2\sqrt{2}} \nu |A^{\frac{3}{2}} u|^2 \leq \sqrt{2}\frac{|A^{\frac{1}{2}} g|^2}{\nu} + \frac{8\sqrt{2}(2c^2_L + c_A)^2}{\nu} |u| |Au|^3 .\nonumber 
\end{align*}
Using Poincar\'{e}'s  inequality and the bound on $|A^{\frac{1}{2}}u|$ obtained in Lemma \ref{lemma54} we obtain
 \begin{equation} 
 \label{maineq2} 
 \frac{d}{d\rho}|Au(t_0 + \rho e^{i \theta})|^2  + \frac{1}{2\sqrt{2}} \nu |A^{\frac{3}{2}} u|^2 \leq \sqrt{2}\frac{|A^{\frac{1}{2}} g|^2}{\nu} + 8\sqrt{2}(2c^2_L + c_A)^2\Rt{1} |Au|^3 . 
 \end{equation}
As before, we ignore the term containing $|A^{\frac{3}{2}} u|^2$
to get the inequality 
\begin{align} 
\label{alpha21}
\frac{d\phi_2(\rho)}{d\rho} \leq \gamma_2 + \beta_2 (\phi_2(\rho))^{\frac{3}{2}}, 
\end{align}
where 
\begin{equation*}
\phi_2(\rho) = |Au(t_0 + \rho e^{i \theta})|^2, \quad 
\gamma_2 = \sqrt{2}\frac{\Rt{1}^2\nu\ko^2}{\delta_1^2},
\quad
\beta_2 = 8\sqrt{2}(2c^2_L +c_A)^2\Rt{1}.
\end{equation*}

From (\ref{alpha21}), we obtain the analogue of the relation (\ref{alpha100}), namely
\begin{equation*}
\frac{2}{\beta_2( (\gamma_2/\beta_2)^{2/3} + \phi_2(0))^{\frac{1}{2}}} - \frac{2}{\beta_2( (\gamma_2/\beta_2)^{2/3} + \phi_2(\rho))^{\frac{1}{2}}} \leq \rho. 
\end{equation*}
We observe that if
\begin{equation*}
\rho < \frac{1}{\beta_2( (\gamma_2/\beta_2)^{2/3} + \phi_2(0))^{\frac{1}{2}}},
\end{equation*}
then 
\begin{equation*}  
((\gamma_2/\beta_2)^{2/3} + \phi_2(\rho))^{\frac{1}{2}} \leq 2 ((\gamma_2/\beta_2)^{2/3} + \phi_2(0))^{\frac{1}{2}} 
\end{equation*}
and hence,
 \begin{align*} 
 |Au(t_0 + \rho e^{i \theta})|^2 \leq 3 (\gamma_2/\beta_2)^{2/3} + 4 |Au(t_0)|^2. 
 \end{align*} 
By (\ref{alpha13}), we obtain  
 \begin{align*}
|Au(t_0 + \rho e^{i \theta})|^2 &\leq \left(\frac{3(\sqrt{2}\cdot 16^2\cdot 24^6c_L^{16})^{2/3}}{4(2c_L^2+c_A)^{4/3}}G^6+4\R{2}^2\right)\nu^2\ko^4,
 \end{align*}
Thus, if we define $\delta_2$ by (\ref{delta2}) and $\Rt{2}$ by (\ref{rt2}), then we obtain (\ref{e542}).
 \end{proof}

We now consider the case $\alpha=3$ after which we can proceed by induction for all  $\alpha >3$. 
Let $\delta_3 = \delta_2/2$, where $\delta_2$ is defined as in Lemma \ref{lemma58} and let $r= (2\sqrt{2} - \sqrt{5}) \delta_3$.
Then, given any $\zeta$ in $\mathcal{S}(\delta_3)$, there is a real $t_0$ such that $D(\zeta, r)$ is in the sector of $D(t_0, 2\sqrt{2}\delta_3)$ where $\theta$ varies from $-\pi/4$ to $\pi/4$, as shown in the next figure. 
	\begin{center}
    	\begin{tikzpicture}[x=2cm,y=2cm] 
        \coordinate (A) at (0,0);
        \draw [step=0.5cm, gray, very thin] (-0.4,-1.0) grid (1.9, 1.0);
        \draw (A) -- (1,1);
        \draw (A) -- (1,-1);
        \draw [very thick] (-0.4, -0.5) -- (1.9, -0.5);
        \draw [very thick] (-0.4, 0.5) -- (1.9, 0.5);
        \draw (1,-1) arc (-45:45:1.415926);
        \draw [gray, thin] (0,0) -- (1.265, 0.6325);
        \draw [gray, thin] (0,0) -- (1.265, -0.6325);
        \node [red] at (1.2, 0.6) {$r$};
        \foreach \x in {-0.5,-0.3,...,0.6}
            {
                \draw (1,\x) circle (0.2962);
                \node at (1,\x) [shape=circle, draw, fill=black, inner sep = 0.1pt] {};
            }
    \coordinate (B) at (0.2,0);
    \coordinate (C) at (0.2,0.2);
    \draw (A) node[left] {$t_0$};
    \def\angleRadius{10pt}
    \draw[red,->,thick] let \p1=(A), \p2=(B), \p3=(C), 
        \n1={atan2(\x2-\x1,\y2-\y1)}, \n2={atan2(\x3-\x1,\y3-\y1)} in
        ($(\p1)!\angleRadius!(\p2)$) arc (\n1:\n2:\angleRadius);
    \draw[red] let \p1=(A), \p2=(B), \p3=(C), 
         \n1={atan2(\x2-\x1,\y2-\y1)}, \n2={atan2(\x3-\x1,\y3-\y1)} in
         (\p1)+(\n1/2+\n2/2:\angleRadius) node[above] {$\frac{\pi}{4}$}; 
    \coordinate (B) at (0.2,0);
    \coordinate (C) at (0.2,-0.2); 
    \def\angleRadius{10pt}
    \draw[red,->,thick] let \p1=(A), \p2=(B), \p3=(C), 
        \n1={atan2(\x2-\x1,\y2-\y1)}, \n2={atan2(\x3-\x1,\y3-\y1)} in
        ($(\p1)!\angleRadius!(\p2)$) arc (\n1:\n2:\angleRadius);
    \draw[red] let \p1=(A), \p2=(B), \p3=(C), 
         \n1={atan2(\x2-\x1,\y2-\y1)}, \n2={atan2(\x3-\x1,\y3-\y1)} in
         (\p1)+(\n1/2+\n2/2:\angleRadius) node[below] {$-\frac{\pi}{4}$};        
    \end{tikzpicture} 
    \end{center}
Using (\ref{maineq2}) and the notation from (\ref{alpha21}) we obtain, for $\theta\in [-\pi/4, \pi/4]$,
\begin{align} 
\label{ineq} 
\frac{\nu}{2\sqrt{2}}\int_0^{2\sqrt{2}\delta_3} |A^{\frac{3}{2}}(t_0+\rho e^{i\theta})|^2d\rho &\leq |Au(t_0)|^2 +\int_{0}^{2\sqrt{2}\delta_3}(\gamma_2 + \beta_2 |Au|^3) d\rho \\
&\leq N_2(\nu\ko^2)^2,\nonumber
\end{align} 
where 
\begin{equation}
\label{N}
N_2:=\R{2}^2+\frac{2\delta_2\Rt{1}^2}{\delta_1^2\nu\ko^2}+16(2c_L^2+c_A)^2\Rt{1}\Rt{2}^3\delta_2\nu\ko^2.
\end{equation}

Using the mean value theorem for the analytic function $A^{\frac{3}{2}}u(\zeta)$ (as we did before for $Au(\zeta)$) in $D(\zeta, r)$ we obtain
\begin{align*} 
|A^{\frac{3}{2}}u(\zeta)| &\leq \frac{1}{\pi r^2}\iint_{\{s_1+is_2 \in D(\zeta, r)\}}|A^{\frac{3}{2}}u(s_1+is_2)|ds_1ds_2\\
&\leq\frac{1}{\pi r^2} \int_{-\pi/4}^{\pi/4}\int_{0}^{2\sqrt{2}\delta_3} |A^{\frac{3}{2}}u(t_0+\rho e^{i\theta})|\rho d\rho d\theta \nonumber\\
&\leq\frac{1}{\pi r^2} \int_{-\pi/4}^{\pi/4} d\theta\left(\int_{0}^{2\sqrt{2}\delta_3} |A^{\frac{3}{2}}u(t_0+\rho e^{i\theta})|^2d\rho\right)^{\frac{1}{2}}\left(\int_0^{2\sqrt{2}\delta_3}\rho^2 d\rho\right)^{\frac{1}{2}} \nonumber\\
&\leq \frac{1}{\pi r^2}\frac{\pi}{2}(2\sqrt{2}N_2\nu\ko^4)^{\frac{1}{2}}\frac{(2\sqrt{2}\delta_3)^{\frac{3}{2}}}{\sqrt{3}}\nonumber\\
&= \frac{4}{\sqrt{3}(2\sqrt{2}-\sqrt{5})^2}N_2^{\frac{1}{2}}\frac{\nu^{\frac{1}{2}}\ko^2}{\delta_3^{\frac{1}{2}}}<4N_2^{\frac{1}{2}}\frac{\nu^{\frac{1}{2}}\ko^2}{\delta_3^{\frac{1}{2}}}.\nonumber
\end{align*}

Now to obtain for $A^{\frac{3}{2}}u(t)$, $t\in\mathbb{R}$, an estimate analogous to (\ref{alpha13}), we use an argument similar to that involving the polygon $abcdef$ in the proof of Lemma \ref{lemma54}; we note that now the roles of $\delta_1$ and (\ref{M}) are played by $\delta_3$ and (\ref{ineq}), respectively.

In this manner, we obtain
\begin{equation}
\label{e560}
|A^{\frac{3}{2}}u(t)|\leq R_3\nu\ko^3,
\end{equation}
where
\begin{equation}
\label{r3}
	R_3 := \frac{12\sqrt{2}}{\pi}(\frac{N_3}{\delta_3\nu\ko^2})^{\frac{1}{2}},
\end{equation}
and
\begin{equation*}
	N_3 := R_2^2+\frac{2\delta_3\Rt{1}^2}{\delta_1^2\nu\ko^2}+16(2c_L^2+c_A)^2\Rt{1}\Rt{2}^3\delta_3\nu\ko^2.
\end{equation*}
We sum up the results obtained above in the following
\begin{lem}
\label{lemma59}
If $0\in\mathcal{A}$ and if $u(t), t\in \mathbb{R}$ is any solution of the NSE in $\mathcal{A}$, then $u(t)$ can be extended to a $\mathcal{D}(A^{\frac{3}{2}})_{\mathbb C}$-valued analytic function $u(\zeta)$, for $\zeta \in \mathcal{S}(\delta_3)$, where
\begin{equation*}
\label{delta3}
\delta_3 := \frac{\delta_2}{2},
\end{equation*}
and $\delta_2$ is defined as in (\ref{delta2}), for which the following estimates holds
\begin{equation*} 
|A^{\frac{3}{2}}u(\zeta)| \leq \Rt{3}\nu\ko^3,  \quad{\mbox{
for $\zeta \in \mathcal{S}(\delta_3)$, }}
\end{equation*}
where 
\begin{equation}
\label{rt3}
\Rt{3} := 4\frac{N_2^{\frac{1}{2}}}{\delta_3^{\frac{1}{2}}\nu^{\frac{1}{2}}\ko},
\end{equation}
 and $N_2$ is defined in (\ref{N}).\\
 
 Moreover, $u(t)$ satisfies the relation (\ref{e560}).
\end{lem}

\begin{remark}
Lemmas \ref{lemma54}, \ref{lemma58} and \ref{lemma59} establish the validity of Theorem \ref{t2} for the case $\alpha\in\{1, 2, 3\}$.  
\end{remark}
\section{Two more Estimates}
\label{s-6}
In this section we present an extension of the estimates given in Lemma \ref{bioplist} and Lemma \ref{lemma51} to the powers $A^{\alpha}(\alpha\in\mathbb{Z}, \alpha>3)$ of $A$ using the method of Constantin (see \cite{CHEN94}), but with $\ln (\ko^{-1}A^{\frac{1}{2}})$ replacing $A^{\theta}$ with $\theta\in (0,1)$.

\begin{lem}
\label{l75}
    Let $u\in\D(A^{\frac{\alpha}{2}}), v\in\D(A^{\frac{\alpha+1}{2}})$, $w\in\D(A^{\alpha})$, and $\alpha > 3$, then
	\begin{align*}
        |(B(u,v),A^{\alpha}w)|&\leq{2^\alpha c_A}\left(|u|^{\frac{1}{2}}|Au|^{\frac{1}{2}}|A^{\frac{1+\alpha}{2}}{v}|+|A^{\frac{\alpha}{2}}{u}||A^{\frac{1}{2}}{v}|^{\frac{1}{2}}|A^{\frac{3}{2}}{v}|^{\frac{1}{2}}\right)|A^{\frac{\alpha}{2}}{w}|.\nonumber
	\end{align*}
\end{lem}
\begin{proof}
	Fix $\alpha >3$. To simplify the exposition, we denote $\tilde{u}:=A^{\frac{\alpha}{2}}u$ and $u\in\D(A^{\frac{\alpha}{2}})$. 
	
	Then for any $u\in\D(A^{\frac{\alpha}{2}}), v\in\D(A^{\frac{\alpha+1}{2}})$, $w\in\D(A^{\alpha})$, we have
		\begin{align*}
		|(B(u,v),A^{\alpha}w)| &\leq L^2\ko^{1+2\alpha}\sum_{\substack{h,j,k\in\mathbb{Z}^2\setminus\{0\}\\h+j+k=0}} |\hat{u}(h)||j||\hat{v}(j)||\hat{w}(k)||k|^{2\alpha}\\
					&= L^2\ko^{1+\alpha}\sum_{\substack{h,j,k\in\mathbb{Z}^2\setminus\{0\}\\h+j+k=0}}|\hat{u}(h)||j||\hat{v}(j)||\hat{\tilde{w}}(k)||k|^{\alpha}\\
					&\leq L^2\ko^{1-\alpha}\sum_{\substack{h,j,k\in\mathbb{Z}^2\setminus\{0\}\\h+j+k=0}}|\hat{\tilde{u}}(h)||j||\hat{\tilde{v}}(j)||\hat{\tilde{w}}(k)|(|h|+|j|)^{\alpha}|h|^{-\alpha}|j|^{-\alpha}\\
					&= L^2\ko^{1-\alpha}\sum_{\substack{h,j,k\in\mathbb{Z}^2\setminus\{0\}\\h+j+k=0}}|\hat{\tilde{u}}(h)||j||\hat{\tilde{v}}(j)||\hat{\tilde{w}}(k)|e^{\alpha[\ln(|h|+|j|)-\ln|h|-\ln|j|]}\\
&=L^2\ko^{1-\alpha}\sum_{\substack{h,j,k\in\mathbb{Z}^2\setminus\{0\}\\h+j+k=0\\|h|\leq|j|}}\cdots +L^2\ko^{1-\alpha}\sum_{\substack{h,j,k\in\mathbb{Z}^2\setminus\{0\}\\h+j+k=0\\|h|>|j|}}\cdots\\
					&=: I_1 + I_2,
	\end{align*}
	where the notation is self-explanatory.
	
  For $I_1$, since $\ln(|h|+|j|)-\ln|h|-\ln|j|$ is decreasing with respect to $|j|$, we have
        \begin{align*}
            I_1 &\leq L^2\ko^{1-\alpha} \sum_{\substack{h,j,k\in\mathbb{Z}^2\setminus\{0\}\\h+j+k=0\\|h|\leq|j|}}|\hat{\tilde{u}}(h)||j||\hat{\tilde{v}}(j)||\hat{\tilde{w}}(k)|e^{\alpha[\ln2-\ln|h|]}\\
            &= L^2\ko \sum_{\substack{h,j,k\in\mathbb{Z}^2\setminus\{0\}\\h+j+k=0\\|h|\leq|j|}}|\hat{u}(h)||j||\hat{\tilde{v}}(j)||\hat{\tilde{w}}(k)|e^{\alpha\ln 2} \\
            &\leq 2^{\alpha}L^2\ko \sum_{\substack{h,j,k\in\mathbb{Z}^2\setminus\{0\}\\h+j+k=0}}|\hat{u}(h)||j||\hat{\tilde{v}}(j)||\hat{\tilde{w}}(k)|
        \end{align*}
    
     By estimating $I_2$ in a same way, we obtain
            \begin{align*}
        		|(B(u,v),A^{\alpha}w)| &\leq 2^\alpha L^2\ko\sum_{\substack{h,j,k\in\mathbb{Z}^2\setminus\{0\}\\h+j+k=0}}[|\hat{u}(h)||j||\hat{\tilde{v}}(j)|+|\hat{\tilde{u}}(h)||j||\hat{v}(j)|]|\hat{\tilde{w}}(k)|
            \end{align*}
    
     We define the auxiliary functions $U$ and $\tilde{U}$ by $U:=\sum_{\substack{k\in\mathbb{Z}^2\setminus\{0\}}}|\hat{u}(k)|e^{i\ko k\cdot x}$ and $\tilde{U}:=\sum_{\substack{k\in\mathbb{Z}^2\setminus\{0\}}}|\hat{\tilde{u}}(k)|e^{i\ko k\cdot x}$ and in a similar way, the functions $V,\tilde{V},W,\tilde{W}$.Then we have that
        \begin{align*}
        		|(B(u,v),A^{\alpha}w)| &\leq {2^\alpha}\int_{[0,L]^2}[(U\cdot(-\Delta)^{\frac{1}{2}}\tilde{V})+(\tilde{U}\cdot(-\Delta)^{\frac{1}{2}}{V})]\tilde{W}d^2x\\
                &\leq {2^\alpha}[|U|_{L^\infty}|(-\Delta)^{\frac{1}{2}}\tilde{V}|_{L^2}+|\tilde{U}|_{L^2}|(-\Delta)^{\frac{1}{2}}{V}|_{L^\infty}]|\tilde{W}|_{L^2}
                \end{align*}
                
        Using  Agmon's inequality, we obtain
        \begin{align*}
        		|(B(u,v),A^{\alpha}w)| \leq& {2^\alpha c_A} [|U|_{L^2}^{\frac{1}{2}}|(-\Delta)U|_{L^2}^{\frac{1}{2}}|(-\Delta)^{\frac{1}{2}}\tilde{V}|_{L^2}\\
        		&+|\tilde{U}|_{L^2}|(-\Delta)^{\frac{1}{2}}{V}|_{L^2}^{\frac{1}{2}}|(-\Delta)^{\frac{3}{2}}{V}|_{L^2}^{\frac{1}{2}}]|\tilde{W}|_{L^2}\\
        =& {2^\alpha c_A}\left(|u|^{\frac{1}{2}}|Au|^{\frac{1}{2}}|A^{\frac{1+\alpha}{2}}{v}|+|A^{\frac{\alpha}{2}}{u}||A^{\frac{1}{2}}{v}|^{\frac{1}{2}}|A^{\frac{3}{2}}{v}|^{\frac{1}{2}}\right)|A^{\frac{\alpha}{2}}{w}|.
        \end{align*}
\end{proof}
\begin{lem}
\label{l46}
    Let $u\in\D(A^{\frac{\alpha}{2}})_\mathbb{C}, v\in\D(A^{\frac{\alpha+1}{2}})_\mathbb{C}$, $w\in\D(A^{\alpha})_\mathbb{C}$ and $\alpha > 3$, then 
	\begin{equation*}
        |(B(u,v),A^{\alpha}w)|\leq {2^{\alpha+\frac{3}{2}} c_A}\left(|u|^{\frac{1}{2}}|Au|^{\frac{1}{2}}|A^{\frac{1+\alpha}{2}}{v}|+|A^{\frac{\alpha}{2}}{u}||A^{\frac{1}{2}}{v}|^{\frac{1}{2}}|A^{\frac{3}{2}}{v}|^{\frac{1}{2}}\right)|A^{\frac{\alpha}{2}}{w}|.
	\end{equation*}
\end{lem}

The proof of Lemma \ref{l46} is omitted since it is similar to the one above, and the only difference is in the constant.

\section{Induction}
The standing assumption in this section is that $0\in\mathcal{A}$. 

Under this assumption we will obtain, for all $\alpha>3$ and for any solution $u(t)$, $t\in\mathbb{R}$, in $\mathcal{A}$, estimates of the form
\begin{equation}
\label{e71}
	|A^\frac{\alpha}{2} u(t)| \leq \R{\alpha}\nu\ko^{\alpha}, \quad \mbox{for $t\in\mathbb{R}$},
\end{equation}
and for its analytic extension 
\begin{equation}
\label{e72}
	|A^\frac{\alpha}{2} u(\zeta)| \leq \Rt{\alpha}\nu\ko^{\alpha}, \quad \mbox{for $\zeta\in \mathcal{S}(\delta_\alpha)$},
\end{equation}
where $\delta = \delta_3$ (i.e., $\delta_\alpha = \delta_3$ for all $\alpha \geq 3$).

Note that (\ref{e71}), (\ref{e72}) were already established for $\alpha=1, 2, 3$ in Section 5. For the values of $\R{\alpha}$, $\Rt{\alpha}$ see (\ref{R1t}), (\ref{R2}), (\ref{rt2}), (\ref{r3}), (\ref{rt3}). Therefore we will assume that (\ref{e71}) and (\ref{e72}) are valid for some $\alpha\in\mathbb{N}, \alpha\geq 3$ and prove by induction (starting with $\alpha=3$) that they are also valid for $\alpha+1$.
%
%
For this we need the following
\begin{lem}
\label{l71}
	Let $u(t)$ be the solution of the NSE satisfying $u(0)=0$  and let $u(\zeta)$ be its $\D(A^{\frac{\alpha}{2}})_\mathbb{C}$-analytic extension in $\mathcal{S}(\delta_\alpha)$, then 
    \begin{equation*}
    g \in \D(A^{\frac{\alpha}{2}}), 
    \end{equation*} 
    and 
   \begin{equation}
   \label{e74}
    	G_{\alpha+1}:=\frac{|A^{\frac{\alpha}{2}}g|}{\nu^2\ko^{\alpha+2}}\leq\frac{\Rt{\alpha}}{\nu\ko^2\delta}.
   \end{equation}
\end{lem}
\begin{proof}
	Since $0\in \mathcal{A}$ , we have $g = \frac{du}{dt}|_{t=0} \in \D(A^{\frac{\alpha}{2}})$, 
	and since $u(\zeta)$ satisfies (\ref{e72}) for $t\in\mathbb{R}$, also that
		\begin{equation*}
			|A^{\frac{\alpha}{2}}\frac{du(\zeta)}{d\zeta}|_{\zeta=t}|=|\frac{1}{2\pi i}\int_{|\xi-t|=\delta}\frac{A^{\frac{\alpha}{2}}u(\xi)}{(\xi-t)^2}d\xi|\leq \frac{\Rt{\alpha}\nu\ko^{\alpha}}{\delta},
		\end{equation*}
		for $t\in\mathbb{R}$, in particular $t=0$.
		Therefore, $G_{\alpha+1}\leq \frac{\Rt{\alpha}}{\nu\ko^2\delta}$.
\end{proof}

\begin{lem}
\label{l72}
Let $u(t), t\in \mathbb{R}$ be any solution of the NSE in $\mathcal{A}$.
If $u(t)$ satisfies (\ref{e71}), and its analytic extension satisfies (\ref{e72}) for some $\alpha \geq 3$, then (\ref{e71}) also holds for $\alpha+1$, i.e.
    \begin{equation}
        |A^{\frac{\alpha+1}{2}}u(t)|\leq \R{\alpha+1}\nu\ko^{\alpha+1},\quad \forall\  t\in\mathbb{R},
    \end{equation}
    where 
    \begin{equation}
    \label{e7-10}
    	\R{\alpha+1}^2 := \frac{36}{\pi^2}\left(\frac{1}{\delta\nu\ko^2}+\frac{4}{\nu^2\ko^4\delta^2}+ 2\sqrt{2}\Gamma_{\alpha}\right)\Rt{\alpha}^2.
    \end{equation}
and $\Gamma_\alpha$ is defined in (\ref{e711}) and (\ref{gamma}).
    
    Moreover, we have
    \begin{equation*}
    	\R{\alpha+1}^2 > \Rt{\alpha}^2, 
    \end{equation*}
\end{lem}
\begin{proof}
	Let $t_0\in\mathbb{R}$ be arbitrary and $\rho\in [0, \sqrt{2}\delta)$.
	
	Once again, we follow the procedure outlined in Remark \ref{r52}. So, taking inner product in both sides of (\ref{annse}) with $A^{\alpha}u$, we get 
		\begin{align*}
			\frac{1}{2}\frac{d}{d\rho}|A^{\frac{\alpha}{2}}u(t_{0}+\rho e^{i\theta})|^2 &\leq |(g, A^{\alpha}u)|-\nu\cos\theta|A^{\frac{\alpha+1}{2}}u|^2 + |(B(u,u),A^{\alpha}u)|.
		\end{align*}

 	For $\alpha = 3$, since
 	 \begin{align*}
 	&(B(u, u), A^3u) = \sum_{j=1,2} [(B(D_ju, u), D_jA^2u) + (B(u, D_ju), D_jA^2u)]\\
 	&=\sum_{j=1,2} [-(B(D_jD_ju, u), A^2u) - 2(B(D_ju, D_ju), A^2u)-(B(u, D_jD_ju), A^2u)],
 \end{align*}
 we have 
  \begin{align*}
 	|(B(u, u), A^3u)| \leq& 4c_L^2|A^{\frac{1}{2}}u|^{\frac{1}{2}}|Au||A^{\frac{3}{2}}u|^{\frac{1}{2}}|A^2u|+8c_L^2|A^{\frac{1}{2}}u|^{\frac{1}{2}}|A^{\frac{3}{2}}u|^{\frac{1}{2}}|Au||A^2u|\\
 	&+4c_L^2|u|^{\frac{1}{2}}|A^{\frac{1}{2}}u|^{\frac{1}{2}}|A^{\frac{3}{2}}u|^{\frac{1}{2}}|A^2u|^{\frac{3}{2}}\nonumber\\
 		\leq& 16c_L^2|u|^{\frac{1}{2}}|A^{\frac{1}{2}}u|^{\frac{1}{2}}|A^{\frac{3}{2}}u|^{\frac{1}{2}}|A^2u|^{\frac{3}{2}}\nonumber\\
 		\leq& \frac{16c_L^2}{\ko^{1/2}}|A^{\frac{1}{2}}u||A^{\frac{3}{2}}u|^{\frac{1}{2}}|A^2u|^{\frac{3}{2}},\nonumber
 \end{align*}
	and consequently
 	\begin{align}
 	\label{e710}
 		\frac{1}{2}&\frac{d}{d\rho}|A^{\frac{3}{2}}u(t_{0}+\rho e^{i\theta})|^2+\frac{3\nu\cos\theta}{4}|A^2u|^2\leq \frac{1}{\nu\cos\theta}|Ag|^2+|(B(u,u),A^{3}u)|\\
 		&\leq \frac{|Ag|^2}{\nu\cos\theta} + \frac{1}{4}((\frac{3}{\nu\cos\theta})^{3/4}\frac{16c_L^2}{\ko^{\frac{1}{2}}}|A^{\frac{1}{2}}u||A^{\frac{3}{2}}u|^{\frac{1}{2}}|)^4 + \frac{3}{4}((\frac{\nu\cos\theta}{3})^{3/4}|A^2u|^{\frac{3}{2}})^{4/3}\nonumber\\
 		&\leq \frac{|Ag|^2}{\nu\cos\theta} + \frac{3^3\cdot 2^{\frac{31}{2}} c_L^8}{\nu^3\ko^2}|A^{\frac{1}{2}}u|^4|A^{\frac{3}{2}}u|^2+ \frac{\nu\cos\theta}{4}|A^2u|^{2}.\nonumber
 	\end{align}
	
	It follows that 
	\begin{align*}
				\frac{1}{2}\frac{d}{d\rho}|A^{\frac{3}{2}}u(t_{0}+\rho e^{i\theta})|^2&+\frac{1}{2}\nu\cos\theta|A^{2}u|^2 \leq \frac{2}{\nu\sqrt{2}}|Ag|^2+\Gamma_{3}\nu\ko^2|A^{\frac{3}{2}}{u}|^2,
			\end{align*}
where 
	\begin{equation}
	\label{e711}
		\Gamma_3 := 3^3\cdot 2^{\frac{31}{2}} c_L^8\Rt{1}^2.
	\end{equation}
	
		 For $\alpha>3$,  by Young's inequality and Lemma \ref{l46}, we obtain
			\begin{align}
			\label{e7-9}
				\frac{1}{2}\frac{d}{d\rho}|A^{\frac{\alpha}{2}}u(t_{0}+\rho e^{i\theta})|^2&+\frac{3}{4}\nu\cos\theta|A^{\frac{\alpha+1}{2}}u|^2 \\
				\leq& \frac{1}{\nu\cos\theta}|A^{\frac{\alpha-1}{2}}g|^2+{2^{\alpha+\frac{3}{2}} c_A}|u|^{\frac{1}{2}}|Au|^{\frac{1}{2}}|A^{\frac{1+\alpha}{2}}{u}||A^{\frac{\alpha}{2}}{u}|\nonumber\\
				&+{2^{\alpha+\frac{3}{2}} c_A}|A^{\frac{1}{2}}{u}|^{\frac{1}{2}}|A^{\frac{3}{2}}{u}|^{\frac{1}{2}}|A^{\frac{\alpha}{2}}{u}|^2, \nonumber
			\end{align}
			and 
			\begin{align*}
				\frac{1}{2}\frac{d}{d\rho}|A^{\frac{\alpha}{2}}u(t_{0}+\rho e^{i\theta})|^2&+\frac{1}{2}\nu\cos\theta|A^{\frac{\alpha+1}{2}}u|^2 \\
				\leq& \frac{1}{\nu\cos\theta}|A^{\frac{\alpha-1}{2}}g|^2+\frac{1}{\nu\cos\theta}\left({2^{\alpha+\frac{3}{2}} c_A}|u|^{\frac{1}{2}}|Au|^{\frac{1}{2}}|A^{\frac{\alpha}{2}}{u}|\right)^2\\
				&+{2^{\alpha+\frac{3}{2}} c_A}|A^{\frac{1}{2}}{u}|^{\frac{1}{2}}|A^{\frac{3}{2}}{u}|^{\frac{1}{2}}|A^{\frac{\alpha}{2}}{u}|^2\\
				\leq&\frac{1}{\nu\cos\theta}|A^{\frac{\alpha-1}{2}}g|^2+\left(\frac{\{2^{\alpha+\frac{3}{2}} c_A\}^2}{\nu\cos\theta}\Rt{1}\nu\Rt{2}\nu\ko^2+{2^{\alpha+\frac{3}{2}} c_A}\sqrt{\Rt{1}\nu\ko\Rt{3}\nu\ko^3}\right)|A^{\frac{\alpha}{2}}{u}|^2\\ 
				\leq& \frac{2}{\nu\sqrt{2}}|A^{\frac{\alpha-1}{2}}g|^2+\Gamma_{\alpha}\nu\ko^2|A^{\frac{\alpha}{2}}{u}|^2,
			\end{align*}
where 
	\begin{equation}
	\label{gamma}
	\Gamma_\alpha:=2^{\alpha+\frac{3}{2}} c_A[2^{\alpha+2}c_A\Rt{1}\Rt{2}+\sqrt{\Rt{1}\Rt{3}}].
	\end{equation}

	It is not hard to prove that $\Gamma_\alpha \geq 1$ for all $\alpha \geq 3$.
	
			 Since $\cos\theta\geq \frac{\sqrt{2}}{2}$, for $\theta\in[-\pi/4, \pi/4]$, we obtain 
	\begin{align} 
	\label{e714}
		\frac{1}{2}\frac{d}{d\rho}|A^\frac{\alpha}{2} u(t_{0}+\rho e^{i\theta})|^2 &+ \nu \frac{\sqrt{2}}{4}|A^{\frac{\alpha+1}{2}}u(\zeta)|^2\leq \frac{\sqrt{2}}{\nu}|A^{\frac{\alpha-1}{2}}g|^2+\nu\ko^2\Gamma_{\alpha}|A^\frac{\alpha}{2} u|^2.  
	\end{align}
  
        It follows that 
            \begin{align*}
    		\nu \frac{\sqrt{2}}{2}\int_0^{\sqrt{2}\delta}|A^{\frac{\alpha+1}{2}}u(t_{0}+\rho e^{i\theta})|^2d\rho
                \leq& |A^\frac{\alpha}{2} u(t_{0})|^2+\frac{4\delta}{\nu}|A^{\frac{\alpha-1}{2}}g|^2\\
                &+2\nu\ko^2\Gamma_{\alpha}\int_0^{\sqrt{2}\delta}|A^\frac{\alpha}{2} u(\zeta)|^2|_{\zeta=t_{0}+\rho e^{i\theta}}d\rho\\
                \leq& \R{\alpha}^2\nu^2\ko^{2\alpha}+4\delta G_\alpha^2 \nu^3\ko^{2(\alpha+1)}+2\nu\ko^2\Gamma_{\alpha}\sqrt{2}\delta\Rt{\alpha}^2\nu^2\ko^{2\alpha}
            \end{align*}
            
        and 
            \begin{align*}
    		\int_0^{\sqrt{2}\delta}|A^{\frac{\alpha+1}{2}}u(t_{0}+\rho e^{i\theta})|^2d\rho
                &\leq \sqrt{2}\R{\alpha}^2\nu\ko^{2\alpha}+4\sqrt{2}\delta G_\alpha^2 \nu^2\ko^{2(\alpha+1)}+ 4\Gamma_{\alpha}\delta\Rt{\alpha}^2\nu^2\ko^{2(\alpha+1)}\\
                &=:N_\alpha.
            \end{align*} 
            
    Since $u(\zeta)$ is $\D(A^{\frac{\alpha}{2}})_\mathbb{C}$-valued analytic in $\mathcal{S}(\delta_\alpha)$,  we obtain (as we have done in the proof of Lemma \ref{lemma54}) the following successive relations  
    \begin{align*}
        |A^{\frac{\alpha+1}{2}}u(t_0)| &= |\frac{1}{\pi\delta^2}\iint_{D(t_0,\delta)}A^{\frac{\alpha+1}{2}}u(\zeta)d\Im(\zeta)d\Re(\zeta)|\\
        &\leq \frac{1}{\pi\delta^2}\iint_{D(t_0,\delta)}|A^{\frac{\alpha+1}{2}}u(\zeta)| d\Im(\zeta)d\Re(\zeta)\\
        &\leq \frac{1}{\pi\delta^2}\iint_{abcdef}|A^{\frac{\alpha+1}{2}}u(\zeta)| d\Im(\zeta)d\Re(\zeta)\\
        &=\frac{2}{\sqrt{2}\pi\delta^2}\int_{t_0-2\delta}^{t_0+\delta}dt\int_0^{\sqrt{2}\delta}|A^{\frac{\alpha+1}{2}}u(\zeta)||_{\zeta=\tau e^{i\pi/4}+t}d\tau\\
        &\leq\frac{2}{\sqrt{2}\pi\delta^2}\int_{t_0-2\delta}^{t_0+\delta}dt\left(\int_0^{\sqrt{2}\delta}|A^{\frac{\alpha+1}{2}}u(\zeta)|^2|_{\zeta=\tau e^{i\pi/4}+t}d\tau\right)^{\frac{1}{2}}(\sqrt{2}\delta)^{\frac{1}{2}}\\
        &\leq \frac{2}{\sqrt{2}\pi\delta^2}3\delta (N_\alpha\sqrt{2}\delta)^{\frac{1}{2}}\\
        &=\frac{6\cdot 2^{\frac{1}{4}}}{\sqrt{2}\pi}(\frac{N_\alpha}{\delta})^{\frac{1}{2}},
    \end{align*}

	that is, using (\ref{e74}),
		\begin{align*}
			|A^{\frac{\alpha+1}{2}}u(t_0)|^2 &\leq \frac{18\sqrt{2}}{\pi^2}\frac{N_\alpha}{\delta}\\
			&=\frac{18\sqrt{2}}{\pi^2}\left(\frac{\sqrt{2}\R{\alpha}^2\nu\ko^{2\alpha}}{\delta}+4\sqrt{2}G_\alpha^2 \nu^2\ko^{2(\alpha+1)}+ 4\Gamma_{\alpha}\Rt{\alpha}^2\nu^2\ko^{2(\alpha+1)}\right)\\
			&\leq \frac{36}{\pi^2}\nu^2\ko^{2(\alpha+1)}\left(\frac{\R{\alpha}^2}{\delta\nu\ko^2}+4\frac{\Rt{\alpha-1}^2}{\nu^2\ko^4\delta^2}+ 2\sqrt{2}\Gamma_{\alpha}\Rt{\alpha}^2\right)\\
			&\leq \frac{36}{\pi^2}\nu^2\ko^{2(\alpha+1)}\left(\frac{1}{\delta\nu\ko^2}+\frac{4}{\nu^2\ko^4\delta^2}+ 2\sqrt{2}\Gamma_{\alpha}\right)\Rt{\alpha}^2\\
			&=:\R{\alpha+1}^2\nu^2\ko^{2(\alpha+1)},
		\end{align*}
		where
		\[
			\R{\alpha+1}^2 := \frac{36}{\pi^2}\left(\frac{1}{\delta\nu\ko^2}+\frac{4}{\nu^2\ko^4\delta^2}+ 2\sqrt{2}\Gamma_{\alpha}\right)\Rt{\alpha}^2.
		\]

		It is easy to check that
				\[
					\R{\alpha+1}>\Rt{\alpha}.
				\] 
				
 	We conclude the proof by observing that $t_0\in\mathbb{R}$ is arbitrary.
\end{proof}

We next extend the result in Lemma \ref{l72} to the strip $\mathcal{S}(\delta_\alpha)$.
\begin{lem}
\label{l73}
If $0\in\mathcal{A}$ and if the solution $u(t), t\in\mathbb{R}$ of the NSE in $\mathcal{A}$ satisfies (\ref{e71}), (\ref{e72}) for $\alpha$, then $u(\zeta)$ is a $\D(A^{\frac{\alpha+1}{2}})_\mathbb{C}$-valued analytic function, and (\ref{e72}) also holds for $\alpha+1$.  In particular, we have
	\begin{equation*}
        |A^{\frac{\alpha+1}{2}}u_\kappa(\zeta)|\leq \Rt{\alpha+1}\nu\ko^{\alpha+1},\quad\forall\  \zeta\in \mathcal{S}(\delta_\alpha),
    \end{equation*}
    where 
    \begin{equation}
    \label{e7-11}
    \Rt{\alpha+1}^2 := \beta^{\Gamma_{\alpha+1}}\frac{36\sqrt{2}}{\pi^2}	       \Gamma_{\alpha}(1+\epsilon_\alpha)\Rt{\alpha}^2,
    \end{equation}
 
 \begin{equation*}
 \beta:=e^{2\sqrt{2}\delta\nu\ko^2},
 \end{equation*} 

 \begin{equation*}
 \epsilon_\alpha=\frac{1}{2\sqrt{2}\Gamma_{\alpha}\delta\nu\ko^2}+\frac{\sqrt{2}}{\Gamma_{\alpha}\nu^2\ko^4\delta^2}+\frac{\pi^2}{72\nu^2\ko^4\delta^4\Gamma_{\alpha}\Gamma_{\alpha+1}},
 \end{equation*}
 and          $\Gamma_\alpha$ is defined in (\ref{gamma}).
 
    Moreover, the following inequality holds
    \begin{equation*}
    	\Rt{\alpha+1}^2 > \R{\alpha+1}^2.
    \end{equation*}
\end{lem}
 \begin{proof} 
	Let $t_0\in\mathbb{R}$ be arbitrary and $\rho\in [0, \sqrt{2}\delta)$.
	By virtue of Remarks \ref{r52} and \ref{ra4}, we can assume that $u(\zeta)$ is $\D(A^{\frac{\alpha+1}{2}})_{\mathbb{C}}$-valued analytic. 
	Taking the inner product  of (\ref{comnseq}) with $A^{\alpha+1}u$, as in the proof of Lemma \ref{l72}, we get 

	\begin{align}
	\label{e7-7}
		\frac{1}{2}\frac{d}{d\rho}|A^\frac{\alpha+1}{2} u(t_{0}+\rho e^{i\theta})|^2 &+ \nu \frac{\sqrt{2}}{4}|A^{\frac{\alpha+2}{2}}u(\zeta)|^2\\
		&\leq \frac{\sqrt{2}}{\nu}|A^{\frac{\alpha}{2}}g|^2+\nu\ko^2\Gamma_{\alpha+1}|A^\frac{\alpha+1}{2} u|^2, \nonumber
	\end{align}
  where the Lemma \ref{l71} is used.
  
 It follows that
						\begin{align*}
								\frac{d}{d\rho}|A^\frac{\alpha+1}{2} u(t_0+\rho e^{i\theta})|^2 \leq \frac{4}{\nu\sqrt{2}}|A^{\frac{\alpha}{2}}g|^2+2\nu\ko^2\Gamma_{\alpha+1}|A^\frac{\alpha+1}{2} u|^2.
						\end{align*}
						
 	Since $\rho\in[0, \sqrt{2}\delta)$, we have, by (\ref{e74}),
						\begin{align*}
								|A^\frac{\alpha+1}{2} u(\zeta)|^2|_{\zeta=t_0+\rho e^{i\theta}}&\leq e^{2\nu\ko^2\Gamma_{\alpha+1}\rho}|A^\frac{\alpha+1}{2} u(t_0)|^2+\frac{\frac{4}{\nu\sqrt{2}}|A^{\frac{\alpha}{2}}g|^2}{2\nu\ko^2\Gamma_{\alpha+1}}(e^{2\nu\ko^2\Gamma_{\alpha+1}\rho}-1)\\
								&\leq e^{2\nu\ko^2\Gamma_{\alpha+1}\rho}\left[|A^\frac{\alpha+1}{2} u(t_0)|^2+\frac{\sqrt{2}|A^{\frac{\alpha}{2}}g|^2}{\nu^2\ko^2\Gamma_{\alpha+1}}\right]\\
								&\leq e^{2\sqrt{2}\delta\nu\ko^2\Gamma_{\alpha+1}}\left[
								\R{\alpha+1}^2\nu^2
								\ko^{2(\alpha+1)}+\frac{\sqrt{2}}{\Gamma_{\alpha+1}}G^2_{\alpha+1}\nu^2\ko^{2\alpha+2}\right]\\
								&\leq \beta^{\Gamma_{\alpha+1}}\left\{\frac{36}{\pi^2}\left[\frac{1}{\delta\nu\ko^2}+\frac{4}{\nu^2\ko^4\delta^2}+ 2\sqrt{2}\Gamma_{\alpha}\right]+\frac{\sqrt{2}}{\nu^2\ko^4\delta^2\Gamma_{\alpha+1}}\right\}\Rt{\alpha}^2\nu^2\ko^{2\alpha+2}\\
								&=\beta^{\Gamma_{\alpha+1}}\frac{72\sqrt{2}}{\pi^2}\Gamma_{\alpha}(1+\epsilon_\alpha)\Rt{\alpha}^2\nu^2\ko^{2\alpha+2}\\
								&=:\Rt{\alpha+1}^2\nu^2\ko^{2\alpha+2},
						\end{align*}
						where 
						\begin{equation}
						\label{ee721}
						\Rt{\alpha+1}^2:=\beta^{\Gamma_{\alpha+1}}\frac{72\sqrt{2}}{\pi^2}\Gamma_{\alpha}(1+\epsilon_\alpha)\Rt{\alpha}^2.
						\end{equation}
\end{proof}

While (\ref{ee721}) shows that $\Rt{\alpha}$ increases with $\alpha$, the next result provides an explicit upper bound.
\begin{prop}
\label{c73}
	For $\alpha>3$,
	\begin{equation*}
		\Rt{\alpha+1}^2\leq C(g){\beta_1}^{4^{\alpha+1}}\beta_2^{(\alpha+1)^2+\frac{9}{2}(\alpha+1)},
	\end{equation*}
	where 
	\begin{align*}
	&C_1:=\prod_{\alpha=3}^{\infty}(1+\epsilon_\alpha), \quad C_3:={4}[2^{\frac{5}{2}}c_A^2\Rt{1}\Rt{2}+2^{\frac{1}{2}}c_A\sqrt{\Rt{1}\Rt{3}}],\\
	&C_2:=3^32^{-7}c_L^8\Rt{1}^2\prod_{\gamma=3}^{\infty}(1+\eta_\gamma), \quad \eta_\gamma=\frac{\sqrt{\Rt{1}\Rt{3}}}{2^{\gamma+2}c_A\Rt{1}\Rt{2}},\\
	&\beta_1:=\beta^{C_3}, \quad \beta_2:=\max\{\frac{72\sqrt{2}}{\pi^2},c_A^2\Rt{1}\Rt{2}\},
	\end{align*}
	and 
	\begin{equation*}
		C(g):=C_1C_2\Rt{3}^2\beta_2^{-19/2}.
	\end{equation*}
\end{prop}
\begin{proof}
	Since $\sum_{\alpha=3}^{\infty}\epsilon_\alpha$ is convergent, we have $C_1:=\prod_{\alpha=3}^{\infty}(1+\epsilon_\alpha)<\infty$.
	 Due to the definition of $\Gamma_\alpha$ in (\ref{gamma}), we have
				\begin{align*}
					\prod_{\gamma=3}^{\alpha}\Gamma_\gamma&= 3^32^{\frac{31}{2}}c_L^8\Rt{1}^2\prod_{\gamma=4}^{\alpha}2^{2\gamma+7/2}(c_A^2\Rt{1}\Rt{2})\left(1+\frac{\sqrt{\Rt{1}\Rt{3}}}{2^{\gamma+2}c_A\Rt{1}\Rt{2}}\right)\\
					&=: 3^32^{\frac{31}{2}}c_L^8\Rt{1}^2\prod_{\gamma=4}^{\alpha}2^{2\gamma+7/2}(c_A^2\Rt{1}\Rt{2})[1+\eta_\gamma]\\
					&= 3^32^{\frac{31}{2}}c_L^8\Rt{1}^22^{\alpha^2+\frac{9}{2}\alpha-\frac{45}{2}}(c_A^2\Rt{1}\Rt{2})^{\alpha-3}\prod_{\gamma=4}^{\alpha}[1+\eta_\gamma]\\
					&<3^32^{-7}c_L^8\Rt{1}^22^{\alpha^2+\frac{9}{2}\alpha}(c_A^2\Rt{1}\Rt{2})^{\alpha-3}\prod_{\gamma=4}^{\infty}[1+\eta_\gamma]\\
					&=:2^{\alpha^2+\frac{9}{2}\alpha}(c_A^2\Rt{1}\Rt{2})^{\alpha-3}C_2,
				\end{align*}
				and
				\begin{align*}
					\sum_{\gamma=4}^{\alpha+1}\Gamma_\gamma&= \sum_{\gamma=4}^{\alpha+1}[2^{2\gamma+7/2}c_A^2\Rt{1}\Rt{2}+2^{\gamma+\frac{3}{2}}c_A\sqrt{\Rt{1}\Rt{3}}]\\
					&=2^{7/2}c_A^2\Rt{1}\Rt{2}\sum_{\gamma=4}^{\alpha+1}2^{2\gamma}+2^{\frac{3}{2}}c_A\sqrt{\Rt{1}\Rt{3}}\sum_{\gamma=4}^{\alpha+1}2^\gamma\\
					&= 2^{7/2}c_A^2\Rt{1}\Rt{2}\frac{4^{\alpha+2}-4^4}{3}+2^{\frac{3}{2}}c_A(2^{\alpha+2}-2^4)\sqrt{\Rt{1}\Rt{3}}\\
					&\leq \left(2^{\frac{5}{2}}c_A^2\Rt{1}\Rt{2}+2^{\frac{1}{2}}c_A\sqrt{\Rt{1}\Rt{3}}\right){4^{\alpha+2}}\\
					&=:C_34^{\alpha+1}.
				\end{align*}
		 It follows from the recursion relation (\ref{ee721}) that
				\begin{align*}
					\Rt{\alpha+1}^2&:=\beta^{\Gamma_{\alpha+1}}\frac{72\sqrt{2}}{\pi^2}\Gamma_{\alpha}(1+\epsilon_\alpha)\Rt{\alpha}^2\\
					&=\beta^{\sum_{\gamma=4}^{\alpha+1}\Gamma_\gamma}(\frac{72\sqrt{2}}{\pi^2})^{\alpha-2}\prod_{\gamma=3}^\alpha\Gamma_\gamma \prod_{\gamma=3}^\alpha(1+\epsilon_\gamma)\Rt{3}^2\\
					&\leq\beta^{C_34^{\alpha+1}}(\frac{72\sqrt{2}}{\pi^2})^{\alpha-2}2^{\alpha^2+\frac{9}{2}\alpha}(c_A^2\Rt{1}\Rt{2})^{\alpha-3}C_2C_1\Rt{3}^2\\
					&\leq \beta^{C_34^{\alpha+1}}C_2C_1\Rt{3}^2\max\{\frac{72\sqrt{2}}{\pi^2}, 2, c_A^2\Rt{1}\Rt{2}\}^{\alpha^2+\frac{13}{2}\alpha-5}\\
					&=: C(g){\beta_1}^{4^{\alpha+1}}\beta_2^{(\alpha+1)^2+\frac{9}{2}(\alpha+1)}.
				\end{align*}
\end{proof}
\begin{remark}
\label{r75}
	Theorem \ref{t2} is now a direct consequence of Lemmas \ref{lemma54}, \ref{lemma58}, \ref{lemma59}, \ref{l73} and Proposition \ref{c73}.
\end{remark}

\begin{remark}
\label{r76}
	The explicit upper bound of $\tilde{R}_{\alpha}$ given in Proposition \ref{c73} provides information about the possible dependence of $\tilde{R}_{\alpha}$ upon $\alpha$; On the other hand, $\delta_{\alpha}=\delta_3$, for all $\alpha >3$.
\end{remark}

\section{The Class $\mathcal{C}(\sigma)$}
The estimates for $\{|A^{\frac{\alpha}{2}}u(t)|, t\in\mathbb{R}\}$ can be slightly improved by shrinking the width $\delta_\alpha$ of the strip $\mathcal{S}(\delta_\alpha)$ in the induction argument for $\alpha>3$.  
\begin{thm}
\label{t714}
	Let $0\in \mathcal{A}$ and let 
	\begin{equation*}
	\delta_{\alpha+1} :=\frac{\delta_{\alpha}}{2}
	\end{equation*}
	 for $\alpha\in\mathbb{N}, \alpha\geq 3$. Then for any solution in $u(t)\in\mathcal{A}, t\in\mathbb{R}$,  one has
	 \begin{equation}
	 \label{e82}
	 	|A^{\frac{\alpha+1}{2}}u(\zeta)|\leq \Rt{\alpha+1}\nu\ko^{\alpha+1}, \quad\forall\  \zeta\in \mathcal{S}(\delta_{\alpha+1}),
	 \end{equation}
	 for all $\alpha\geq 3$, where the constants $\Rt{\alpha+1}, \alpha \geq 3$, are redefined in the following way
	\begin{equation}
	\label{e83}
		\Rt{\alpha+1}^2 := \frac{1024\sqrt{2}}{\pi^2}\Gamma_{\alpha}(1+\xi_\alpha)\Rt{\alpha}^2,
	\end{equation}
	with
	\begin{equation*}
	\xi_\gamma = \frac{1}{4\sqrt{2}\nu\ko^2\delta_{\alpha+1}\Gamma_\alpha}+\frac{1}{\sqrt{2}\nu^2\ko^4\delta_\alpha\delta_{\alpha+1}\Gamma_\alpha}.
	\end{equation*}
	Furthermore, we have the following estimate
	\begin{equation}
	\label{e84}
		\Rt{\alpha+1}^2\leq \tilde{C}(g)\beta_3^{\frac{3}{2}(\alpha+1)^2},
	\end{equation}	
	where \\
	\begin{equation*}
	\beta_3:=\max(\frac{1024\sqrt{2}}{\pi^2}, c_A^2\Rt{1}\Rt{2}), \quad \tilde{C}(g):=C_2C_4\Rt{3}^2\beta_3^{-3/8},
	\end{equation*}
	and
	\begin{equation*}
	C_4:=\prod_{\gamma=3}^{\infty}(1+\xi_\gamma).
	\end{equation*} 
\end{thm}
\begin{proof}
	As done in the proof of Lemma \ref{lemma59}, we can easily prove that under the new definition (\ref{e83}), the relation (\ref{e82}) is true.
	
	Then, we obtain (as in the proof of Proposition \ref{c73}) 
				\begin{align*}
					\Rt{\alpha+1}^2&:=\frac{1024\sqrt{2}}{\pi^2}\Gamma_{\alpha}(1+\xi_\alpha)\Rt{\alpha}^2\\
					&<(\frac{1024\sqrt{2}}{\pi^2})^{\alpha-2}\Rt{3}^2\prod_{\gamma=3}^\alpha\Gamma_\gamma \prod_{\gamma=3}^\infty(1+\xi_\gamma)\\
					&\leq (\frac{1024\sqrt{2}}{\pi^2})^{\alpha-3}2^{\alpha^2+\frac{9}{2}\alpha}(c_A^2\Rt{1}\Rt{2})^{\alpha-2}C_2C_4\Rt{3}^2\\
					&\leq \tilde{C}(g)\beta_3^{\frac{3}{2}(\alpha+1)^2}.
				\end{align*}
\end{proof}

The estimates in Theorem \ref{t714} identify the role of the subset of $C^\infty([0, L]^2)\cap H$ defined below.
\begin{definition}
\label{d81}
\begin{equation*}
\mathcal{C}(\sigma):=\{u\in C^\infty([0, L]^2)\cap H: \exists\, c_0 = c_0(u)\in\mathbb{R}  \text{ such that } \frac{|A^{\frac{\alpha}{2}}u|^2}{\nu^2\ko^{2\alpha}} \leq c_0e^{\sigma\alpha^2}, \alpha\in\mathbb{N}\}.
 \end{equation*}
\end{definition}

\begin{remark}
The main conclusion of Theorem \ref{t714} can be given in the following succinct formulation
\begin{equation*}
	0\in\mathcal{A} \Rightarrow \mathcal{A} \subset \mathcal{C}(\frac{3}{2}\ln\beta_3).
\end{equation*}

\end{remark}
\begin{remark}
	An equivalent definition of the class $\mathcal{C}(\sigma)$ is 
	\begin{equation}
	\label{def-norm}
		\mathcal{C}(\sigma) = \{u\in C^\infty([0, L]^2)\cap H: |u|_{\mathcal{C}_\sigma}:=\sup\{|A^{\frac{\alpha}{2}}u| e^{-\frac{\sigma}{2}\alpha^2}, \alpha\in\mathbb{N}\} < \infty\}.
	\end{equation}
	It is easy to check that $u\mapsto |u|_{\mathcal{C}_\sigma}$ is a norm on $\mathcal{C}(\sigma)$. Obviously, $\mathcal{C}(\sigma)$ equipped with this norm is a Banach space.
\end{remark}

Moreover, Theorem \ref{t714} has the following corollary
\begin{cor}
\label{t8-2}
	If $0\in\mathcal{A}$, then $g\in\mathcal{C}(\frac{5}{2}\ln \beta_3).$
\end{cor}
\begin{proof}
	 Since $\delta_\alpha=\frac{1}{2^{\alpha-3}}\delta_3>\frac{1}{2^{\alpha}}\delta_3$, by (\ref{e74}) and (\ref{e84}) we get
				\begin{align*}
					|A^{\frac{\alpha}{2}}g|&\leq \Rt{\alpha}\nu\ko^{\alpha}\frac{2^\alpha}{\delta_3}\leq \Rt{\alpha}\nu\ko^{\alpha}\frac{\beta_3^\alpha}{\delta_3}\leq \frac{\Rt{\alpha}\nu\ko^{\alpha}}{\delta_3\beta_3^{-\frac{1}{2}}}\beta_3^{\frac{1}{2}\alpha^2}.
				\end{align*}
				and then
				\begin{equation*}
					\frac{|A^{\frac{\alpha}{2}}g|^2}{\nu^4\ko^{2\alpha+4}} \leq \frac{\tilde{C}(g)}{\nu^2\ko^4\delta_3^2\beta_3^{-1}}\beta_3^{\frac{5}{2}\alpha^2},					
				\end{equation*}
				where both sides are dimensionless.
				
				Consequently, it follows that
				\begin{equation*}
					g\in \mathcal{C}(\frac{5}{2}\ln\beta_3).
				\end{equation*}
\end{proof}
\begin{remark}
We can prove that if $v\in H$ satisfies
\begin{equation*}
	|e^{b(\ln(\ko^{-1}A^{\frac{1}{2}}+a))^2}v| < \infty, \quad a > e, \quad b>0
\end{equation*}
then
\begin{equation}
\label{ee813}
	v \in \mathcal{C}({1}/{b}).
\end{equation}

Indeed, noting the following relation
\begin{align*}
	|A^{\frac{\alpha}{2}}v| =& 	|A^{\frac{\alpha}{2}}e^{-b(\ln(\ko^{-1}A^{\frac{1}{2}}+a))^2}e^{b(\ln(\ko^{-1}A^{\frac{1}{2}}+a))^2}v| \\
	\leq&|A^{\frac{\alpha}{2}}e^{-b(\ln(\ko^{-1}A^{\frac{1}{2}}+a))^2}|_\textrm{op}|e^{b(\ln(\ko^{-1}A^{\frac{1}{2}}+a))^2}v|,
\end{align*}
since 
\begin{align*}
	|A^{\frac{\alpha}{2}}e^{-b(\ln(\ko^{-1}A^{\frac{1}{2}}+a))^2}u|^2 =& \sum_{k\in \mathbb{Z}^2\setminus\{0\}} |k|^{2\alpha}e^{-2b(\ln(|k|+a))^2}|\hat{u}(k)|^2\\
	\leq& \sup_{k\in \mathbb{Z}^2\setminus\{0\}}|k|^{2\alpha}e^{-2b(\ln(|k|+a))^2}\sum_{k\in \mathbb{Z}^2\setminus\{0\}} |\hat{u}(k)|^2\\
	=&\sup_{k\in \mathbb{Z}^2\setminus\{0\}}|k|^{2\alpha}e^{-2b(\ln(|k|+a))^2}|u|^2,
\end{align*}
we obtain that
\begin{align*}
	|A^{\frac{\alpha}{2}}e^{-b(\ln(\ko^{-1}A^{\frac{1}{2}}+a))^2}|^2_\textrm{op} \leq& \sup_{k\in \mathbb{Z}^2\setminus\{0\}}|k|^{2\alpha}e^{-2b(\ln(|k|+a))^2}\\
	\leq& \sup_{k\in \mathbb{Z}^2\setminus\{0\}}(|k|+a)^{2\alpha}e^{-2b(\ln(|k|+a))^2}\\
	\leq& \sup_{1+a\leq x}e^{2\alpha\ln x}e^{-2b(\ln x)^2}
	= e^{\frac{\alpha^2}{2b}}.
\end{align*}
and (\ref{ee813}) follows from Definition \ref{d81}.

\end{remark}

\section{Proof Of Proposition \ref{p23}}\label{proof23sec}
In this section we will use the short procedure given in Remark \ref{r52} (see also Remark \ref{ra4}). Namely, we assume that the solution $u(t)$ and its analytic extension $u(\zeta)$ exists and then establish the necessary a priori estimates. In addition, for simplicity, we use the following notation
\begin{equation*}
	\rho_\beta := \rho_{\max}(G, \frac{|A^{\frac{\beta}{2}}u_0|}{\nu\ko^\beta}),
\end{equation*}
\begin{equation}
\label{mab}
	M_{\alpha, \beta} := M_\alpha(G, G_{\alpha-1},\frac{|A^{\frac{\beta}{2}}u_0|}{\nu\ko^\beta}),
\end{equation}
where 
\begin{equation*}
	G_\alpha : = \frac{|A^{\frac{\alpha}{2}}g|}{\nu\ko^{\alpha+2}},
\end{equation*}
$\rho_{\max}(\cdot, \cdot)$ is defined in (\ref{eea14}) and $M_\alpha(\cdot, \cdot)$ is defined in (\ref{a18}), (\ref{m2}), (\ref{m3}) and (\ref{mb}).

Using Corollary \ref{c2} and Lemma \ref{la1}, we easily obtain
\begin{prop}
\label{p10}
	$u(\zeta)$ is $\D(A^{\frac{1}{2}})_\mathbb{C}$-valued analytic function in $\Pi(t_0, \rho_1)$ and satisfies
\begin{equation*}
	|A^{\frac{1}{2}} u(\zeta)| \leq [M_{1,1}^2+\sqrt{2}G^2]^{\frac{1}{2}}\nu\ko, 
\end{equation*}
for any $\zeta\in \Pi(t_0, \rho_1)$, where $\Pi(\cdot,\cdot), M_{1,1}$ are defined in (\ref{a19}) and (\ref{mab}).
\end{prop}

To prove Proposition \ref{p23}, we need to prove first the following
\begin{prop}
\label{p11}
	Assume $\beta\in\mathbb{N}$, $u_0\in\D(A^{\frac{\beta}{2}})$, and $g\in\D(A^{\frac{\beta-1}{2}})$ then $u$ is $\D(A^{\frac{\beta+1}{2}})_\mathbb{C}$-valued analytic function in the interior of $\sector(t_0, \rho_\beta)$ and there exists $M_{\beta,\beta}$ such that  
	\begin{equation}
	\label{e9-2}
		|A^{\frac{\beta}{2}}u(\zeta)| \leq M_{\beta,\beta}\nu\ko^\beta, \quad \zeta\in \sector(t_0, \rho_\beta),
\end{equation}
holds, 		
where $\sector(\cdot, \cdot)$ is defined in (\ref{sector}) and $M_{\beta,\beta}$ is defined in (\ref{mab}).
	\end{prop}	
\begin{proof}
 For $\beta = 1$, (\ref{e9-2}) is  a direct corollary of Lemmas \ref{la-1}, \ref{la1} and \ref{lemma54}.
 
 For $\beta = 2$, we will need a supplementary estimate for the term $(B(u, u), A^2u)$. 
 Integrating by parts we have
 \begin{align*}
 	(B(u, u), A^2u) = \sum_{j=1,2} [(B(D_ju, u), D_jAu) + (B(u, D_ju), D_jAu)].
 \end{align*}
 Using interpolation ($|Au|\leq|A^{\frac{1}{2}}u|^{\frac{1}{2}}|A^{\frac{3}{2}}u|^{\frac{1}{2}}$ and $|A^{\frac{1}{2}}u|\leq|u|^{\frac{2}{3}}|A^{\frac{3}{2}}u|^{\frac{1}{3}}$), we infer that
 \begin{align}
 \label{b2}
 	|(B(u, u), A^2u)| &\leq 4c_L^2(|A^{\frac{1}{2}}u||Au||A^{\frac{3}{2}}u|+|u|^{\frac{1}{2}}|A^{\frac{1}{2}}u|^{\frac{1}{2}}|Au|^{\frac{1}{2}}|A^{\frac{3}{2}}u|^{\frac{3}{2}})\\
 		&\leq 8c_L^2|u|^{\frac{1}{2}}|A^{\frac{1}{2}}u|^{\frac{1}{2}}|Au|^{\frac{1}{2}}|A^{\frac{3}{2}}u|^{\frac{3}{2}}\nonumber\\
 		&\leq \frac{8c_L^2}{\ko^{\frac{1}{2}}}|A^{\frac{1}{2}}u||Au|^{\frac{1}{2}}|A^{\frac{3}{2}}u|^{\frac{3}{2}}.\nonumber
 \end{align}
	Using (\ref{b2}), we obtain
 	\begin{align*}
 		\frac{1}{2}\frac{d}{d\rho}|Au(t_{0}+\rho e^{i\theta})|^2&+\frac{3\nu\cos\theta}{4}|A^{\frac{3}{2}}u|^2\leq \frac{1}{\nu\cos\theta}|A^{\frac{1}{2}}g|^2+|(B(u,u),A^{2}u)|\\
 		&\leq \frac{|A^{\frac{1}{2}}g|^2}{\nu\cos\theta} + \frac{1}{4}\left[(\frac{3}{\nu\cos\theta})^{3/4}\frac{8c_L^2}{\ko^{\frac{1}{2}}}|A^{\frac{1}{2}}u||Au|^{\frac{1}{2}}|\right]^4 + \frac{3}{4}((\frac{\nu\cos\theta}{3})^{3/4}|A^{\frac{3}{2}}u|^{\frac{3}{2}})^{4/3}\\
 		&\leq \frac{|A^{\frac{1}{2}}g|^2}{\nu\cos\theta} + \frac{3^3\cdot 2^{{23}/{2}} c_L^8}{\nu^3\ko^2}|A^{\frac{1}{2}}u|^4|Au|^2+ \frac{\nu\cos\theta}{4}|A^{\frac{3}{2}}u|^{2},
 	\end{align*}
 	whence 
 	\begin{align*}
 	\frac{d}{d\rho}|Au(t_{0}+\rho e^{i\theta})|^2&+\frac{\sqrt{2}\nu}{2}|A^{\frac{3}{2}}u|^2\leq \frac{2\sqrt{2}|A^{\frac{1}{2}}g|^2}{\nu} + \frac{3^3\cdot 2^{{25}/{2}} c_L^8}{\nu^3\ko^2}|A^{\frac{1}{2}}u|^4|Au|^2,
 	\end{align*}
since $|\theta|\leq \frac{\pi}{4}$.
 	
 	By Lemmas \ref{la-1} and \ref{la1}, for all $ \zeta\in \sector(t_0, \rho_2)$, we have $|A^{\frac{1}{2}}u(\zeta)|\leq M_{1,2}\nu\ko$. It follows that
 	 \begin{align}
 	 \label{ee95}
 	\frac{d}{d\rho}|Au(t_{0}+\rho e^{i\theta})|^2&+\frac{\sqrt{2}\nu}{2}|A^{\frac{3}{2}}u|^2\leq a + b|Au|^2,
 	\end{align}
 	where $a = {2\sqrt{2}|A^{\frac{1}{2}}g|^2}/{\nu}$ and $b = {3^3\cdot 2^{{25}/{2}} c_L^8\nu\ko^2}M_{1,2}^4$.
 	Then, we obtain
 	\begin{align*}
 		|Au(t_{0}+\rho e^{i\theta})|^2&\leq e^{b\rho}|Au(t_{0})|^2+\frac{a}{b}(e^{b\rho}-1)\leq e^{b\rho}[|Au_0|^2+\frac{a}{b}] \leq e^{b\rho_2}[|Au_0|^2+\frac{a}{b}],
 	\end{align*}
 	that is, 
 	\begin{equation*}
 		|Au(\zeta)| \leq M_{2,2}\nu\ko^2,\quad \forall\ \zeta\in \sector(t_0, \rho_2),
 	\end{equation*}
 	where $M_{2,2} = M_{2}(G, G_1,\frac{|Au_0|}{\nu\ko})$ and 
 	\begin{align}
 	\label{m2}
 		M_2(G, G_1, \frac{|Au_0|}{\nu\ko}) &:= e^{3^3\cdot 2^{{23}/{2}} c_L^8\nu\ko^2M_{1,2}^4\rho_2}\left[\frac{|Au_0|^2}{\nu^2\ko^4}+\frac{G_1^2}{3^3\cdot 2^{10} c_L^8}M_{1,2}^{-4}\right]^{\frac{1}{2}}.
 	\end{align}

	Furthermore, from (\ref{ee95}) we obtain
	\begin{align}
	\label{e98}
		\int_0^{\rho_2}|A^\frac{3}{2}u(t_0+\rho e^{i\theta})|^2d\rho \leq  \gamma,
	\end{align}	 	
	where 
	\begin{align*}
		\gamma &:= \frac{4|A^{\frac{1}{2}}g|^2}{\nu^2}\rho_2+\frac{\sqrt{2}}{\nu}|Au_0|^2
 	+ {3^3\cdot 2^{13} c_L^8\nu\ko^6}M_{1,2}M_{2,2}^2\rho_2.	\\
	\end{align*}
 	For any $\zeta = t_0 + \rho e^{i\theta}$, $|\theta|<\frac{\pi}{4}$ and $0<\rho<\rho_2$, we can choose $r > 0$ such that
 	\begin{equation}
 	\label{e99}
 		D(\zeta, r) \subset \sector(t_0, \rho_2).
 	\end{equation}	
	From (\ref{e98}), by the same method used in the proof of Lemma \ref{lemma54}, we obtain that
	\begin{align}
	\label{e910}
		|A^{\frac{3}{2}}u(\zeta)| &= |\frac{1}{\pi r^2}\int_{D(\zeta, r)}A^{\frac{3}{2}}u(\xi)d\Re(\xi)d\Im(\xi)|\leq \frac{\rho_2^{\frac{3}{2}}\gamma^{\frac{1}{2}}}{2\sqrt{3}r^2}.
	\end{align}	 
	 Therefore, $u(\zeta)$ is a $\D(A^{\frac{3}{2}})_\mathbb{C}$-valued analytic function in the interior of the sector $\sector(t_0, \rho_2)$.	
 	
 	For $\beta = 3$, from (\ref{e710}) we proceed as in the case $\beta = 2$ to obtain
 	\begin{equation}
 	\label{e9-12}
 		|A^{\frac{3}{2}}u(\zeta)| \leq M_{3,3}\nu\ko^3,\quad \forall\ \zeta\in \sector(t_0, \rho_3),
 	\end{equation} 
 	where $M_{3,3} = M_{3}(G, G_2,\frac{|A^{\frac{3}{2}}u_0|}{\nu\ko})$ and 
 	\begin{align}
 	\label{m3}
 		M_3(G, G_2, \frac{|A^{\frac{3}{2}}u_0|}{\nu\ko^3})&:= e^{3^3\cdot 2^{{31}/{2}} c_L^8\nu\ko^2M_{1,3}^4\rho_3}\left[\frac{|A^{\frac{3}{2}}u_0|^2}{\nu^2\ko^6}+\frac{G_2^2}{3^3\cdot 2^{15} c_L^8}M_{1,3})^{-4}\right]^{\frac{1}{2}}.
 	\end{align} 	

	Furthermore, the case $\beta=3$ can be treated as the previous case $\beta=2$ by deducing from (\ref{e9-12}) and (\ref{m3}) the sequence of relations analogous to the sequence of (\ref{e98}), (\ref{e99}) and (\ref{e910}). We obtain that $u(\zeta)$ is a $\D(A)_\mathbb{C}$-valued analytic function in the interior of the sector $\sector(t_0, \rho_3)$.

 	For the cases $\beta > 3$, by Lemma \ref{l46}, we have 
 				\begin{align}
 				\label{e914}
				\frac{1}{2}\frac{d}{d\rho}|A^{\frac{\beta}{2}}u(t_{0}&+\rho e^{i\theta})|^2+\frac{1}{2}\nu\cos\theta|A^{\frac{\beta+1}{2}}u|^2 \leq \frac{|A^{\frac{\beta-1}{2}}g|^2}{\nu\cos\theta}\\
				&+\frac{\{{2^{\beta+\frac{3}{2}} c_A}|u|^{\frac{1}{2}}|Au|^{\frac{1}{2}}|A^{\frac{\beta}{2}}{u}|\}^2}{\nu\cos\theta}+{2^{\beta+\frac{3}{2}} c_A}|A^{\frac{1}{2}}{u}|^{\frac{1}{2}}|A^{\frac{3}{2}}{u}|^{\frac{1}{2}}|A^{\frac{\beta}{2}}{u}|^2.\nonumber
     			\end{align}

 	It follows that for any $\zeta \in \sector(t_0, \rho_\beta)$, we have 
 	\begin{align*}
 	 	 |A^{\frac{1}{2}}u(\zeta)| &\leq M_{1,\beta}\nu\ko,\,
|Au(\zeta)| \leq M_{2,\beta}\nu\ko^2,\,
|A^{\frac{3}{2}}u(\zeta)| \leq M_{3,\beta}\nu\ko^3.
\end{align*}

	Therefore, we obtain
		 \begin{align*}
				\frac{1}{2}\frac{d}{d\rho}|A^{\frac{\beta}{2}}u(t_{0}+\rho e^{i\theta})|^2&+\frac{1}{2}\nu\cos\theta|A^{\frac{\beta+1}{2}}u|^2 \\
				\leq& \frac{|A^{\frac{\beta-1}{2}}g|^2}{\nu\cos\theta}+\frac{1}{\nu\cos\theta}\{{2^{\beta+\frac{3}{2}} c_A}(M_{1,\beta}M_{2,\beta})^{\frac{1}{2}}\nu\ko|A^{\frac{\beta}{2}}{u}|\}^2\\
				&+{2^{\beta+\frac{3}{2}} c_A}(M_{1,\beta}M_{3,\beta})^{\frac{1}{2}}\nu\ko^2|A^{\frac{\beta}{2}}{u}|^2\\
				\leq& \frac{\sqrt{2}}{\nu}|A^{\frac{\beta-1}{2}}g|^2+\gamma_\beta\nu\ko^2|A^{\frac{\beta}{2}}{u}|^2,
     	\end{align*}
 	where 
 	\begin{equation*}
 	\gamma_\beta = 2^{2\beta+7/2}c_A^2M_{1,\beta}M_{2,\beta}+{2^{\beta+\frac{3}{2}} c_A}\sqrt{M_{1,\beta}M_{3,\beta}}.
 	\end{equation*}
 	and consequently 
 	\begin{align*}
 		|A^{\frac{\beta}{2}}u(\zeta)| &\leq M_{\beta,\beta}\nu\ko^\beta, \quad \forall\ \zeta \in \sector(t_0, \rho_\beta),
 	\end{align*}
	where $M_{\beta,\beta} = M_{\beta}(G, G_{\beta-1},\frac{|A^{\frac{\beta}{2}}u_0|}{\nu\ko^\beta})$ and 
 	\begin{align}
 	\label{mb}
 		M_\alpha(G, G_{\alpha-1}, \frac{|A^{\frac{\beta}{2}}u_0|}{\nu\ko^\beta}) :=e^{{\gamma_\beta} \nu\ko^2\rho_\beta}\left[\frac{|A^{\frac{\beta}{2}}u_0|^2}{\nu^2\ko^{2\beta}}+\frac{\sqrt{2}G_{\alpha-1}^2}{\gamma_\beta}\right]^{\frac{1}{2}}, \quad \forall\, \alpha>3.
 	\end{align}

	Furthermore, using the same method as in the case $\beta = 2\text{ or }3$, we can deduce that $u(\zeta)$ is a $\D(A^{\frac{\beta+1}{2}})_\mathbb{C}$-valued analytic function in the interior of this sector $\sector(t_0, \rho_\beta)$.	
\end{proof}

	We are now ready to give the proof of Proposition \ref{p23}. First, the case $\alpha = 1$ is a direct consequence of Lemma \ref{lemma54}. Suppose that Proposition \ref{p23} is valid for some $\alpha\geq 1$. This means that if $g\in\D(A^{\frac{\alpha-1}{2}})$, then $\mathcal{A}\subset\D(A^{\frac{\alpha+1}{2}})$ and moreover any solution $u(\cdot)$ in $\mathcal{A}$ can be extended to an $\D(A^{\frac{\alpha+1}{2}})_\mathbb{C}$-valued analytic function in the strip $\mathcal{S}(\delta_\alpha)$ such that 
	\begin{equation}
	\label{e911}
	\sup\{|A^{\frac{\alpha}{2}}u(\zeta)|: \zeta\in \mathcal{S}(\delta_\alpha)\} \leq m_\alpha\nu\ko^\alpha< \infty.
\end{equation}	 

We will prove, under the assumption $g\in\D(A^{\frac{\alpha}{2}})$, that $\mathcal{A}\subset\D(A^{\frac{\alpha+2}{2}})$ and that any solution $u(\cdot)$ in $\mathcal{A}$ can be extended to a $\D(A^{\frac{\alpha+2}{2}})_\mathbb{C}$-valued analytic function in some strip $\mathcal{S}(\delta_{\alpha+1})$ satisfying 
	\begin{equation}
	\label{e912}
	 \sup\{|A^{\frac{\alpha+1}{2}}u(\zeta)|: \zeta\in \mathcal{S}(\delta_{\alpha+1})\}\leq 		m_{\alpha+1}\nu\ko^{\alpha+1} < \infty.
\end{equation}	 
	
	 The estimate in (\ref{e911}) is in particular valid for $t\in\mathbb{R}$. Therefore, applying Proposition \ref{p11} for $\beta = \alpha+1$ and each $t_0 \in \mathbb{R}$, we obtain by defining
	\begin{equation*}
		m_{\alpha+1} := M_{\alpha+1}(G, G_\alpha, m_\alpha),
	\end{equation*}
	and 
	\begin{equation*}
		\delta_{\alpha+1} := \frac{1}{\sqrt{2}}\rho_{\max}(G, m_\alpha)
\end{equation*}
that the solution $u(\cdot)$ in $\mathcal{A}$ can be extended to a $\D(A^{\frac{\alpha+2}{2}})_\mathbb{C}$-valued analytic function in the strip $\mathcal{S}(\delta_{\alpha+1})$ and relation (\ref{e912}) holds. Thus the proof is completed by induction.

\section{The ``All for one, one for all" law}
\begin{prop}
\label{p91}
	Let $\alpha\in\mathbb{N}$ be fixed. Then
	\[
		\mathcal{A}\cap \D(A^{\frac{\alpha}{2}}) \neq \emptyset \Rightarrow \mathcal{A}\subset \D(A^{\frac{\alpha}{2}}).
	\]
Furthermore, $g\in\D(A^{\frac{\alpha-2}{2}}), \alpha>2$.
\end{prop}

\begin{proof}
	Note that the case $\alpha=1,2$ are trivially true.
   We now consider the case $\alpha=3$. 
		
		Let $u^0\in\mathcal{A}\cap\D(A^{\frac{3}{2}})$ and let $u(t)$ denote the solution of the NSE such that $u(0) = u^0$.
		We already know that this solution extends to a $\D(A)_\mathbb{C}-$ valued analytic solution $u(\zeta)$ in a strip of $\mathcal{S}(\delta_\alpha):=\{\zeta\in\mathbb{C}:|\Im(\zeta)|<\delta\}, \delta>0$. 
		Therefore from  
		\begin{equation*}
			\frac{d}{d\zeta}u(\zeta) + \nu Au(\zeta) + B(u(\zeta), u(\zeta)) = g, \quad \zeta\in \mathcal{S}(\delta_\alpha),
		\end{equation*}
		we obtain
		\begin{equation}
		\label{e92}
			g = \frac{d}{d\zeta}u(\zeta)|_{\zeta=0} + \nu Au^0 + B(u^0, u^0).
		\end{equation}
		where $\frac{d}{d\zeta}u(\zeta)|_{\zeta=0}\in\D(A)$, $\nu Au^0 \in \D(A^{\frac{1}{2}})$.\\
		Since for $\forall\  w\in \D(A)$ we have
		\begin{align*}
			|(B(u^0, u^0), Aw)| &= |\sum_{j=1,2} (B(D_ju^0,u^0)+B(u^0, D_ju^0),D_jw)|\\
			&\leq (c_L^2|A^{\frac{1}{2}}u^0||Au^0|+c_A|u^0|^{\frac{1}{2}}|Au^0|^{\frac{3}{2}})(|D_1w|+|D_2w|)\\
			&\leq \sqrt{2}(c_L^2+c_A)|u^0|^{\frac{1}{2}}|Au^0|^{\frac{3}{2}}|A^{\frac{1}{2}}w|,
		\end{align*}
		we get that
		\begin{equation}
		\label{e93}		
			B(u^0, u^0)\in\D(A^{\frac{1}{2}}),		
		\end{equation}
		and
		\[
				|A^{\frac{1}{2}}B(u^0, u^0)| \leq \sqrt{2}(c_L^2+c_A)|u^0|^{\frac{1}{2}}|Au^0|^{\frac{3}{2}}.
		\]
		
		Thus from (\ref{e92}) and (\ref{e93}) we infer that
		\begin{equation*}
			g\in\D(A^{\frac{1}{2}}).
		\end{equation*}
		Proposition \ref{p23} now yields
		\begin{align*}
			\mathcal{A} &\subset \D(A^{\frac{3}{2}}), 
		\end{align*}
		and that there exists a $\delta_3\in (0, \delta_2)$ such that any solution $u(t), t\in \mathbb{R}$ in $\mathcal{A}$ extends to a $\D(A^{\frac{3}{2}})_\mathbb{C}$- valued analytic function in the strip $\mathcal{S}(\delta_3):=\{\zeta\in\mathbb{C}:|\Im(\zeta)|<\delta_3\}$.
		
		We will proceed now by induction on $\alpha$; the induction assumption will be that for some $\alpha\geq 3$
		\begin{equation}
		\label{e105}
			\mathcal{A}\cap \D(A^{\frac{\alpha}{2}}) \neq \emptyset \Rightarrow \mathcal{A}\subset \D(A^{\frac{\alpha}{2}}).
		\end{equation}
		Assume that there exists a $u^0\in \mathcal{A}\cap \D(A^{\frac{\alpha+1}{2}})$ and that $u(t)$ is the solution of the NSE in $\mathcal{A}$ satisfying $u(0) = u^0$. The induction assumption implies that $Au\in\D(A^{\frac{\alpha-1}{2}})$ and $\frac{du}{dt}|_{t=0} = \frac{du}{d\zeta}|_{\zeta=0} \in \D(A^\frac{\alpha}{2})$.
		
		By Lemma \ref{l75}, we have that for $\forall\  w\in\D(A^{\alpha-1})$
		\begin{align*}
				|(B(u^0,u^0),A^{\alpha-1}w)|&\leq {2^{\alpha-1} c_A}\left(|u^0|^{\frac{1}{2}}|Au^0|^{\frac{1}{2}}|A^{\frac{\alpha}{2}}{u^0}|+|A^{\frac{\alpha-1}{2}}{u^0}||A^{\frac{1}{2}}{u^0}|^{\frac{1}{2}}|A^{\frac{3}{2}}{u^0}|^{\frac{1}{2}}\right)|A^{\frac{\alpha-1}{2}}{w}|.
		\end{align*}
		It follows that
		\[
				B(u^0,u^0)\in \D(A^{\frac{\alpha-1}{2}}).
		\]
		Thus we have that $g\in\D(A^{\frac{\alpha-1}{2}})$.
			Using Proposition \ref{p23} we finally obtain that (\ref{e105}) also holds for $\alpha+1$.
		This completes the proof.
		
\end{proof}

Observe that since
\[
	H \cap C^\infty([0,L]^2; \mathbb{R}^2) = \bigcap_{\alpha\in\mathbb{N}} \D(A^{\frac{\alpha}{2}})\;,
\]
Theorem \ref{t24} is a direct consequence of Proposition \ref{p91}.
\begin{remark}
	Proposition \ref{p91} and Theorem \ref{t24} assert that if there is one point in the attractor $\mathcal{A}$ belonging to a certain class (namely $H\cap H^\alpha([0,L]^2) (\alpha\in\mathbb{N})$ or $H\cap C^\infty([0,L]^2)$) then all points of $\mathcal{A}$ belong to this class. We will show that for the class defined in Definition \ref{d81} the ``All for one, one for all" law is also valid. Moreover, we expect that this law is ``almost'' universal, in that it holds for a variety of subsets of $H$.
\end{remark}

\section{$\bigcup_{\sigma>0}\mathcal{C}(\sigma)\subsetneqq
C^{\infty}([0,L]^2; \mathbb{R}^2)\cap H$}

A natural question one may ask regarding the newly defined classes
$\mathcal{C}(\sigma)$, $\sigma>0$, is: will the union of all the classes
$\mathcal{C}(\sigma)$ for $\sigma$ ranging in $(0, \infty)$ actually be the
same as the family of $C^{\infty}$ functions in $H$? The answer is no.

\begin{thm}
\label{th11-1}
$\bigcup_{\sigma>0}\mathcal{C}(\sigma)\subsetneqq C^{\infty}([0,L]^2;
\mathbb{R}^2)\cap H$
\end{thm}

To prove this theorem, we assume that
$\bigcup_{\sigma>0}\mathcal{C}(\sigma)=C^{\infty}([0,L]^2;\mathbb{R}^2)\cap
H$ and we will arrive at a contradiction. To this end, we need the
following observation:

\begin{prop}
\label{prop11-2}
There exists a 1-1 correspondence between
$\dot{C}^{\infty}_{per}([0,L]^2;\mathbb{R})$ and
$C^{\infty}([0,L]^2;\mathbb{R}^2)\cap H$.
\end{prop}

In the above proposition, $\dot{C}^{\infty}_{per}([0,L]^2;\mathbb{R})$
consists of elements in $C^{\infty}_{per}([0,L]^2;\mathbb{R})$ with
zero-average.

\begin{proof}
Let $\psi$ be any given element in
$\dot{C}^{\infty}_{per}([0,L]^2;\mathbb{R})$. Define $u=(u_1,u_2)$ by
setting $u_1=\frac{\partial\psi}{\partial x_2}$,
$u_2=-\frac{\partial\psi}{\partial x_1}$, then one immediately sees that
$u_1,u_2$ both belong to $C^{\infty}_{per}([0,L]^2;\mathbb{R})$, and
$\nabla \cdot u=0$, thus we get an element $u$ in
$C^{\infty}([0,L]^2;\mathbb{R}^2)\cap H$. (We recall that $\psi$ is called
the stream function corresponding to $u=(u_1,u_2)$).

Conversely, given $u=(u_1,u_2)$, with $u_1,u_2\in
C^{\infty}_{per}([0,L]^2;\mathbb{R})$ and $\nabla \cdot u=0$, we will get
the stream function $\psi$ by uniquely solving
\begin{equation*}
u_1=\frac{\partial\psi}{\partial x_2},  u_2=-\frac{\partial \psi}{\partial
x_1}.
\end{equation*}
From the above equation, we have $\frac{\partial u_1}{\partial
x_2}-\frac{\partial u_2}{\partial x_1}=\Delta \psi$. When we express this
equality in the form of Fourier series expansion, we uniquely determine the
coefficient $\hat{\psi}(k)$ for $\psi$ expressed in terms of the
coefficients of $u_1,u_2$ as follows:
\begin{equation}
\label{stream}
\hat{\psi}(k)=-\frac{k_2 \hat{u_1}(k)-k_1 \hat{u_2}(k)}{|k|^2},
\end{equation}
and hence the stream function corresponding to the function $u=(u_1,u_2)$
is 
\begin{equation*}
\psi(x)=\sum_{k\in \mathbb{Z}^2\setminus\{0\}}\hat{\psi}(k)e^{i\kappa_0
k\cdot x}=\sum_{k\in \mathbb{Z}^2\setminus\{0\}}   -\frac{k_2
\hat{u_1}(k)-k_1 \hat{u_2}(k)}{|k|^2} e^{i\kappa_0 k\cdot x}.
\end{equation*}
Therefore, the one-to-one correspondence between these two families is
established.
\end{proof}

\begin{proof}[Proof of Theorem \ref{th11-1}]
Given an arbitrary $u \in
C^{\infty}([0,L]^2;\mathbb{R}^2)\cap H$, suppose that $u\in
\mathcal{C}(\sigma)$, for some $\sigma>0$. Let $\psi$ be the corresponding
stream function, invoking the relation (\ref{stream}), one can show that
\begin{equation*}
|(-\Delta)^{\frac{\alpha}{2}}\psi|_{L^2}\leq
|A^{\frac{\alpha-1}{2}}u|_{L^2}\leq \sqrt{c_0}e^{\sigma (\alpha-1)^2/2}\nu
\kappa_0^{\alpha-1},
\end{equation*}
where the constant $c_0$ is from the definition of the class.
Equivalently, 
\begin{equation*}
|(-\Delta)^{\frac{\alpha+1}{2}}\psi|_{L^2}\leq \sqrt{c_0}e^{\sigma
\alpha^2/2}\nu \kappa_0^{\alpha}, \quad \forall\  \alpha\geq 0.
\end{equation*}
It follows from the above inequality and the Sobolev embedding theorem
that, 
\begin{equation*}
|(-\Delta)^{\frac{\alpha-1}{2}}\psi|_{L^{\infty}}\leq
c_1|(-\Delta)^{\frac{\alpha-1}{2}}\psi|_{H^2}\leq
c_2|(-\Delta)^{\frac{\alpha+1}{2}}\psi|_{L^2}\leq c_3 e^{\sigma
\alpha^2/2}\nu \kappa_0^{\alpha}, \quad \forall\ \ \alpha\geq 1.
\end{equation*}
where $c_1,c_2,c_3$ are non-dimensional constants.

The one-to-one correspondence in Proposition \ref{prop11-2} and the
inequality given above imply that, for any $\psi \in
\dot{C}^{\infty}_{per}([0,L]^2;\mathbb{R})$, we have
\begin{equation}
\label{star}
|(-\Delta)^{\frac{\alpha-1}{2}}\psi|_{L^{\infty}}\leq c e^{\sigma
\alpha^2/2}\nu \kappa_0^{\alpha},\quad \forall\  \alpha\geq 1.
\end{equation} 
for some constant $c$. Notice that
$|(-\Delta)^{\frac{\alpha-1}{2}}\psi|_{L^{\infty}}$ controls the magnitude
of all the derivatives of the function $\psi$ up to order $\alpha-1$, so
this inequality says that the coefficients in the Taylor series expansion
of every function $\psi \in \dot{C}^{\infty}_{per}([0,L]^2;\mathbb{R})$
must have an estimate of the form $e^{\sigma \alpha^2/2}$.

However, we can easily construct a function $\eta \in
\dot{C}^{\infty}_{per}([0,L]^2;\mathbb{R})$ which does not have this kind of
estimate for the coefficients of its Taylor series expansion. To find such
a $C^{\infty}$ function, we choose $x_0$ to be the center of the square
$[0,L]^2$, and for any given $C^{\infty}$ function $\psi$ consider first
the following power series
\begin{equation*}
\sum_{\beta \in \mathbb{N}^2}\frac{1}{\beta!}e^{\sigma
|\beta|^4/2}(x-x_0)^{\beta}
\end{equation*}
It follows from Borel's Theorem (see page 381 in \cite{R70}) that this series defines a
$C^{\infty}$ function, denoted by $\psi$ on, say the disk $D(x_0,r_0)$, for
some $r_0 < \frac{L}{2\sqrt{2}}$. Moreover, we extend the definition $\psi$
to $\mathbb{R}^2$ by setting $\psi=0$ outside this disk.

Then choose another function $\psi \in C^{\infty}([0,L]^2)$ such that
$\psi=1$ on $D(x_0,r_0/3)$, with $supp(\psi)\subset D(x_0,2r_0/3)$ and
$0\leq \psi \leq 1$. 
Finally, let $\eta_1=\psi\psi$, and define
$\eta=\eta_1-\int_{[0,L]^2}\eta_1(x)dx$, then extend the definition of
$\eta$ to $\mathbb{R}^2$ by periodicity, to get  $\eta \in
\dot{C}^{\infty}_{per}([0,L]^2;\mathbb{R})$. Now the function $\eta$ has
coefficients in its Taylor expansion which contradict the estimates in (\ref{star}). 
Hence, the assumption
$\bigcup_{\sigma>0}\mathcal{C}(\sigma)=C^{\infty}([0,L]^2)\cap H$ is not true,
and the theorem is proved.
\end{proof}

A second question regarding the classes defined in this paper concerns the
possibility that $\mathcal{C}(\sigma_1)=\mathcal{C}(\sigma_2)$, for some
$\sigma_1<\sigma_2$. Equivalently, does the family of the classes
$\{\mathcal{C}(\sigma)\}_{\sigma>0}$ form an increasing (or decreasing)
family? 
We answer this question by proving the following proposition.

\begin{prop}
\label{increasing}
For the family of classes $\{\mathcal{C}(\sigma)\}_{\sigma>0}$, we have,
\begin{equation*}
\mathcal{C}(\sigma_1)\subsetneq \mathcal{C}(\sigma_2), \quad \forall\ \ 
\sigma_1<\sigma_2.
\end{equation*}
\end{prop}

\begin{proof}
We assume by contradiction that 
\begin{equation*}
\mathcal{C}(\sigma_1)= \mathcal{C}(\sigma_2),
\end{equation*}
for some  $\sigma_1<\sigma_2$.
Since both $\mathcal{C}(\sigma_1)$ and $\mathcal{C}(\sigma_2)$ are Banach
spaces, if it happens that these two are the same, then their norms are
comparable; i.e., there exist $m>0, M>0$ such that
\begin{equation}
\label{increasing property}
m\leq
\frac{|u|_{\mathcal{C}_{\sigma_1}}}{|u|_{\mathcal{C}_{\sigma_2}}}\leq M,
\end{equation}
holds for all $u$, where $|\cdot|_{\mathcal{C}_\sigma}$ is defined in (\ref{def-norm}).

Initially, in the definition of the class $\mathcal{C}(\sigma)$, we
allow only  $\alpha\in \mathbb{N}$; however, we are going to show that we could
actually extend this definition by taking $\alpha \in \mathbb{R}$. Indeed,
for any $\alpha <\beta < \alpha+1$, by interpolation we have 
\begin{equation*}
|A^{\frac{\beta}{2}}u|\leq
|A^{\frac{\alpha}{2}}u|^{\alpha+1-\beta}|A^{\frac{\alpha+1}{2}}u|^{\beta-\alpha}.
\end{equation*}
One could easily obtain, using the estimate $|A^{\frac{\alpha}{2}}u|^2\leq
c_0 e^{\sigma \alpha^2}\nu^2 \kappa_0^{2\alpha}$, the following 
\begin{equation*}
|A^{\frac{\beta}{2}}u|^2\leq {c_0}' e^{\sigma \beta^2}
\nu^2\kappa_0^{2\beta}
\end{equation*}
where ${c_0}'=c_0e^{\sigma/4}$. Notice that the change of constant $c_0$
does not depend on the values of $\alpha$ and $\beta$.

Assuming that such a modification in our definition of the class
$\mathcal{C}(\sigma)$ has been made, we take in particular the
function $u(x)=\sum_{\{k:\kappa_0^2 |k|^2=\Lambda\}}\hat{u}(k)e^{i\kappa_0
k\cdot x}$, for some fixed $\Lambda$. After a direct calculation, one gets
\begin{equation*}
\frac{|u|_{\mathcal{C}_{\sigma_1}}}{|u|_{\mathcal{C}_{\sigma_2}}}=e^{\frac{(\ln
\Lambda)^2}{8\sigma_1}-\frac{(\ln
\Lambda)^2}{8\sigma_2}}=e^{\frac{\sigma_2-\sigma_1}{8\sigma_1 \sigma_2}
(\ln \Lambda)^2}.
\end{equation*}
As $\Lambda \rightarrow \infty$, we have that
\begin{equation*}
\frac{|u|_{\mathcal{C}_{\sigma_1}}}{|u|_{\mathcal{C}_{\sigma_2}}}\rightarrow
\infty\;.
\end{equation*}
Thus, we see that the constant $M>0$ in (\ref{increasing property}) cannot
exist.
\end{proof}

\section{``One for all and all for one'' law for $ \bigcup_{\sigma>0}C(\sigma)  $}
\label{sec:oneall}
The main result in this section is simple to state.

\begin{thm}
\label{t121}
	If $ \mathcal{A}\cap \bigcup_{\sigma>0}\mathcal{C}(\sigma) \neq \emptyset$, then $ \mathcal{A} \subset \bigcup_{\sigma>0}\mathcal{C}(\sigma)$.
\end{thm}

The proof is a consequence of the following three lemmas.

\begin{lem}
\label{l12-2}
Let $u(t), t\in \mathbb{R}$ be any solution of the NSE in $\mathcal{A}$.
If $g\in C^\infty([0,L]^2)$, then for any $\alpha\in\mathbb{N}$,  $u$ is a $\D(A^{\frac{\alpha}{2}})_\mathbb{C}$-valued analytic function in some strip $\mathcal{S}(\delta_\alpha)$ and satisfies
	\begin{equation*}
		|A^{\frac{\alpha}{2}}u(\zeta)| \leq M_{\alpha}\nu\ko^\alpha, \,\, \forall\  \zeta\in \mathcal{S}(\delta_\alpha).
	\end{equation*}
	
	In particular, for $\alpha\geq3$, we can choose 	 
\begin{equation}
\label{e12-2}
\delta_{\alpha+1} := \frac{\delta_{\alpha}}{2},  
\end{equation}
to obtain that 
	\begin{equation}
	\label{e12-3}
		M_{\alpha+1} := \frac{8\sqrt{2}}{\pi}\left(\frac{M^2_\alpha}{\delta_\alpha\nu\ko^2}+4G^2_{\alpha-1}+\sqrt{2}\Gamma_\alpha M^2_\alpha\right)^{\frac{1}{2}},
\end{equation}		 
\begin{equation}
\label{e12-4}
	G_\alpha := \frac{|A^{\frac{\alpha}{2}}g|}{\nu^2\ko^{\alpha+2}},
\end{equation}
and 
	\begin{equation} 
	\label{e125}
	\Gamma_\alpha:=
		\begin{cases} 3^3\cdot 2^{{31}/{2}} c_L^8M_1^2 &\text{if } \alpha=3,\\
		2^{\alpha+\frac{3}{2}} c_A[2^{\alpha+2}c_AM_1M_2+\sqrt{M_1M_3}] &\text{otherwise}.
		\end{cases}
	\end{equation}

\end{lem}
\begin{proof}
	Since $g\in C^\infty([0,L]^2)$, using the same argument as in the proof of Lemma \ref{lemma54}, Lemma \ref{lemma58} and Lemma \ref{lemma59}, we know that for $\alpha = 1, 2, 3$, there exist $\delta_\alpha$ and $M_\alpha$ such that $u(\zeta)$ is  $\D(A^{\frac{\alpha}{2}})_\mathbb{C}$-valued analytic in the strip $\mathcal{S}(\delta_\alpha)$ and satisfies the following estimate
\begin{equation*}
	|A^{\frac{\alpha}{2}}u(\zeta)| \leq M_{\alpha}\nu\ko^\alpha, \,\, \forall\ \ \zeta\in \mathcal{S}(\delta_\alpha).
\end{equation*}	
The only difference is that $\delta_2, M_2$, and $M_3$ depend on $G_1$.

	For $\alpha \geq 3$,  we have the following inequality (\ref{e714})
	\begin{align*} 
		\frac{1}{2}\frac{d}{d\rho}|A^\frac{\alpha}{2} u(t_{0}+\rho e^{i\theta})|^2 &+ \nu \frac{\sqrt{2}}{4}|A^{\frac{\alpha+1}{2}}u(\zeta)|^2 \leq \frac{2}{\nu\sqrt{2}}|A^{\frac{\alpha-1}{2}}g|^2+\nu\ko^2\Gamma_{\alpha}|A^\frac{\alpha}{2} u|^2, \nonumber
	\end{align*}
where $\Gamma_\alpha$ is defined as in the Lemma \ref{l72}, and the only difference is that it depends on $M_\alpha$, not $\Rt{\alpha}$, see (\ref{e125}).
 It follows that
	\begin{align*}
		\int_0^{\sqrt{2}\delta_\alpha}|A^{\frac{\alpha+1}{2}}u(\zeta)||_{\zeta=t+\rho e^{i\frac{\pi}{4}}}d\rho &\leq \sqrt{2}\left(M_\alpha+4G^2_{\alpha-1}\nu\ko^2\delta_\alpha +\sqrt{2}\nu\ko^2\delta_\alpha\Gamma_\alpha M_\alpha^2\right)\nu\ko^{2\alpha}.
	\end{align*}
	
	Proceeding as in the proof of Theorem \ref{t714}, i.e. choosing $\delta_{\alpha+1}$ as in (\ref{e12-2}), we obtain that for $\forall\  \zeta\in \mathcal{S}(\delta_{\alpha+1})$, there is a disk $D(\zeta, \delta_{\alpha+1})$ contained in $\mathcal{S}(\delta_\alpha)$. It follows that
	\begin{align*}
		|A^{\frac{\alpha+1}{2}}u(\zeta)| &\leq \frac{2\delta_\alpha}{\pi\delta^2_{\alpha+1}}\int_0^{\sqrt{2}\delta_\alpha}|A^{\frac{\alpha+1}{2}}u(\xi)||_{\xi=t+\rho e^{i\frac{\pi}{4}}}d\rho\\
		&\leq\frac{2\delta_\alpha}{\pi\delta^2_{\alpha+1}}\left(\int_0^{\sqrt{2}\delta_\alpha}|A^{\frac{\alpha+1}{2}}u(\xi)|^2|_{\xi=t+\rho e^{i\frac{\pi}{4}}}d\rho\right)^{\frac{1}{2}}(\sqrt{2}\delta_\alpha)^{\frac{1}{2}}\\
		&\leq \frac{8\sqrt{2}}{\pi}\left(\frac{M_\alpha^2}{\delta_\alpha\nu\ko^2}+4G^2_{\alpha-1}+\sqrt{2}\Gamma_\alpha M^2_\alpha\right)^{\frac{1}{2}}\nu\ko^{\alpha+1}\\
		&= M_{\alpha+1}\nu\ko^{\alpha+1}.
	\end{align*}
	 
\end{proof}
\begin{remark}
Notice that $M_\alpha$ depends only on $g$.
\end{remark}	
\begin{lem}
\label{l124}
If $ \mathcal{A}\cap \mathcal{C}(\sigma) \neq \emptyset $ for some $\sigma >0$, then $g\in C^\infty([0,L]^2)$ and $G_\alpha$ defined in (\ref{e12-4}) satisfies the following estimate
\begin{equation*}
 G_\alpha^2 \leq \gamma_2 e^{\sigma_2\alpha^2}, \, \forall\  \alpha\in\mathbb{N},  \alpha \geq \alpha_1,
\end{equation*}
where  
	\begin{equation*}
		\alpha_1:= \max\{\lfloor \log_4\frac{2}{c_4} \rfloor+1, 4\},
	\end{equation*}
and $\gamma_2, \sigma_2, c_4$ are defined in (\ref{e12-19}), (\ref{e12-20}), (\ref{e12-15}).
 
\begin{proof}
	Let $u_0 \in \mathcal{A}\cap \mathcal{C}(\sigma)$ and $u(t)$ be the solution of the NSE with $u(0)=u_0$.
	By Proposition \ref{p91}, we infer that $g\in C^\infty([0,L]^2)$. Then applying Lemma \ref{l12-2}, we know that $u$ is a $\D(A^{\frac{\alpha}{2}})_\mathbb{C}$-valued analytic function in the strip $\mathcal{S}(\delta_\alpha)$ and satisfies
	\begin{align}
	\label{e12-8}
		|\frac{d}{d\rho} A^{\frac{\alpha}{2}}u(0)|=|\frac{1}{2\pi i}\int_{\partial D(0,\delta_\alpha)}\frac{A^{\frac{\alpha}{2}}u(\xi)}{\xi^2}d\xi| 
		\leq \frac{M_\alpha\nu\ko^\alpha}{\delta_\alpha}.
	\end{align}
	Since $u_0\in \mathcal{C}(\sigma)$, we obtain that there exists $c_0$ such that 
	\begin{align*}
		\frac{|A^{\frac{\alpha}{2}}Au_0|^2}{(\nu\ko^{\alpha+2})^2}\leq c_0e^{\sigma(\alpha+2)^2} \leq c_0 e^{8\sigma} e^{2\sigma\alpha^2}.
	\end{align*}
	That is 
	\begin{align}
	\label{e12-9}
		|A^{\frac{\alpha}{2}}Au_0| \leq c_1e^{\sigma\alpha^2}\nu\ko^{\alpha+2},
	\end{align}
	where 
	\begin{equation*}
		c_1 := \sqrt{c_0}e^{4\sigma}.
	\end{equation*}
	
	Using Lemma \ref{l46}, we have 
	\begin{align}
	\label{e12-11}
		|A^{\frac{\alpha}{2}}B(u_0, u_0)|^2 &\leq 2^{2\alpha}c_A^2\left(|u_0||Au_0||A^{\frac{\alpha+1}{2}}u_0|^2+|A^{\frac{1}{2}}u_0||A^{\frac{3}{2}}u_0||A^{\frac{\alpha}{2}}u_0|^2\right)\\
		&\leq 2^{2\alpha}c_A^2c_0^2\left(e^{\sigma[(\alpha+1)^2+1]}+e^{\sigma(\alpha^2+5)}\right)\nu^4\ko^{2\alpha+4}\nonumber\\
		&\leq 2c_A^2c_0^24^\alpha e^{\sigma[(\alpha+1)^2+1]} \nu^4\ko^{2\alpha+4}\nonumber\\
		&= 2c_A^2c_0^2 e^{2\sigma} e^{\sigma[\alpha^2+\alpha(2+\frac{\ln 4}{\sigma})]} \nu^4\ko^{2\alpha+4}\nonumber\\
		&\leq c_2^2 e^{2\sigma\alpha^2}\nu^4\ko^{2\alpha+4},\nonumber
	\end{align}
	where 
	\begin{equation*}
		c_2 := \sqrt{2}c_Ac_0 e^{\sigma(1+\frac{1}{8}(2+\frac{\ln 4}{\sigma})^2)}.
	\end{equation*}
	From (\ref{e12-8}), (\ref{e12-9}), (\ref{e12-11}), we obtain that
	\begin{align*}
		|A^{\frac{\alpha}{2}}g| &= |\frac{d}{d\rho}A^{\frac{\alpha}{2}}u(0)+\nu A^\frac{\alpha}{2}Au_0+A^\frac{\alpha}{2}B(u_0,u_0)|\\
		&\leq [\frac{M_\alpha}{\delta_\alpha\nu\ko^2}+c_3 e^{\sigma\alpha^2}]\nu^2\ko^{\alpha+2},\nonumber
	\end{align*}
	i.e.
	\begin{equation}
	\label{e12-13}
		G_\alpha \leq \frac{M_\alpha}{\delta_\alpha\nu\ko^2}+c_3 e^{\sigma\alpha^2},
	\end{equation}
	where 
	\begin{equation*}
		 c_3 := c_1 + c_2.
	\end{equation*}

	Applying Lemma \ref{l12-2} and (\ref{e12-13}), we infer that for $\alpha>3$, $M_\alpha$ defined in (\ref{e12-3}) satisfies
	\begin{align*}
		M_{\alpha+1}^2 &\leq \frac{128}{\pi^2}\left(\frac{M^2_\alpha}{\delta_\alpha\nu\ko^2}+4[\frac{M_{\alpha-1}}{\delta_{\alpha-1}\nu\ko^2}+c_3e^{\sigma(\alpha-1)^2}]^2+\sqrt{2}\Gamma_\alpha M_\alpha^2\right)\\
		&\leq \frac{128}{\pi^2}\left(\frac{M^2_\alpha}{\delta_\alpha\nu\ko^2}+8[\frac{M^2_{\alpha-1}}{\delta_{\alpha-1}^2\nu^2\ko^4}+c^2_3e^{2\sigma(\alpha-1)^2}]+\sqrt{2}\Gamma_\alpha M_\alpha^2\right)\\
		&\leq\frac{128M^2_\alpha}{\pi^2}\left(\frac{2^{\alpha-3}}{\delta_3\nu\ko^2}+\frac{2^{2\alpha-3}}{\delta_{3}^2\nu^2\ko^4}+2^{2\alpha+4}c_A^2M_1M_2+2^{\alpha+2}\sqrt{M_1M_3} \right)+\frac{2^8}{\pi^2}c^2_3e^{2\sigma(\alpha-1)^2}\\
		&= c_4 4^{\alpha+1} M^2_\alpha+c_5e^{2\sigma(\alpha-1)^2},\\
	\end{align*}
	where 
	\begin{align}
	\label{e12-15}
		c_4 &=\frac{128}{\pi^2}\left(\frac{1}{\delta_3\nu\ko^2}+\frac{1}{\delta_{3}^2\nu^2\ko^4}+2^2c_A^2M_1M_2+c_A\sqrt{M_1M_3} \right),\quad
		c_5 = \frac{2^8}{\pi^2}c^2_3.
	\end{align}

	It follows that
	\begin{align*} 
		M_{\alpha}^2 &\leq c_4 4^\alpha M^2_{\alpha-1}+c_5e^{2\sigma(\alpha-2)^2}\\
		&\leq c_4 4^\alpha (c_4 4^{(\alpha-1)} M^2_{\alpha-2}+c_5e^{2\sigma(\alpha-3)^2})+c_5e^{2\sigma(\alpha-2)^2}\nonumber\\
		&\leq (c_4 4^\alpha)^2M^2_{\alpha-2}+c_5(1+c_44^\alpha)e^{2\sigma(\alpha-2)^2}\nonumber\\
		&\cdots\nonumber\\
		&\leq (c_44^\alpha)^{\alpha-4}M^2_4+c_5[1+\cdots + (c_44^\alpha)^{\alpha-4}]e^{2\sigma(\alpha-2)^2}\nonumber\\
		&=(c_44^\alpha)^{\alpha-4}M^2_4+c_5\frac{(c_44^\alpha)^{\alpha-3}-1}{(c_44^\alpha)-1}e^{2\sigma(\alpha-2)^2}.\nonumber
	\end{align*}
	
		For $\alpha \geq \alpha_1$ (i.e. $c_44^\alpha\geq 2$) we have 
	\begin{align}
	\label{e1220}
		M_{\alpha}^2 &\leq (c_44^\alpha)^{\alpha-4}M^2_4+2c_5{(c_44^\alpha)^{\alpha-4}}e^{2\sigma(\alpha-2)^2}\\
		&= (M^2_4+2c_5)c_4^{-4}c_4^\alpha 4^{\alpha(\alpha-4)}e^{2\sigma(\alpha-2)^2}\nonumber\\
		&\leq \gamma_1 e^{2\sigma_1\alpha^2},\nonumber
	\end{align}
	where
	\begin{align*} 
		\sigma_1 := \ln 4+ 2\sigma,\quad  \gamma_1 := (M^2_4+2c_5)c_4^{-4} e^{8\sigma}
						e^\frac{\ln c_4-4\ln 4-8\sigma}{4\sigma_1}.	
	\end{align*}
Therefore, by (\ref{e12-13}), we infer that for $\alpha\geq\alpha_1$, $G_\alpha$ satisfies
	\begin{align*}
		G^2_\alpha &\leq [\frac{M_\alpha}{\delta_\alpha\nu\ko^2}+c_3 e^{\sigma\alpha^2}]^2\leq \frac{2M^2_\alpha}{\delta^2_\alpha\nu^2\ko^4}+2c_3^2 e^{2\sigma\alpha^2}\leq \frac{\gamma_1 e^{2\sigma_1\alpha^2}2^{2\alpha-5}}{\delta^2_3\nu^2\ko^4}+2c_3^2 e^{2\sigma\alpha^2}\\
		&\leq \frac{\gamma_1 e^{3\sigma_1\alpha^2}2^{-5}e^{\frac{(\ln 4)^2}{4\sigma_1}}}{\delta^2_3\nu^2\ko^4}+2c_3^2 e^{2\sigma\alpha^2}\leq \gamma_2 e^{\sigma_2\alpha^2},
\end{align*}		 
where
	\begin{align}
	\label{e12-19}
		\gamma_2 := \frac{\gamma_1 2^{-5}e^{\frac{(\ln 4)^2}{4\sigma_1}}}{\delta^2_3\nu^2\ko^4}+2c_3^2,\\
		\label{e12-20}
		 \sigma_2 := \max\{3\sigma_1, 2\sigma\}.
	\end{align}
\end{proof}
\end{lem}
\begin{remark}
	By defining 
	\begin{equation*}
		c_0 := \max\{\gamma_2, \frac{|A^{\frac{1}{2}}g|^2}{\nu^4\ko^6}, \cdots, \frac{|A^{\frac{\alpha_1}{2}}g|^2}{\nu^4\ko^{2\alpha_1+4}}\},
	\end{equation*}
we see that $g\in \mathcal{C}(\sigma_2)$.
\end{remark}

Combining the Lemma \ref{l12-2} and Lemma \ref{l124}, we obtain
\begin{lem}
\label{l126}
If $ \mathcal{A}\cap C(\sigma) \neq \emptyset $ for some $\sigma >0$, then for $u(t), t\in \mathbb{R}$, any solution of the NSE in $\mathcal{A}$, we obtain that
	\begin{equation*}
		|A^{\frac{\alpha}{2}}u(t)| \leq M_{\alpha}\nu\ko^\alpha, \,\, \forall\  t \in\mathbb{R},\,\forall\  \alpha\in\mathbb{N}, \alpha>\alpha_1
	\end{equation*}	
	and 
	\begin{equation*}
		M_\alpha^2 \leq \gamma_3 e^{\sigma_3\alpha^2}, 
	\end{equation*}	 
	where $\sigma_3, \gamma_3$ are defined in (\ref{e1231}), (\ref{e1232}).
\end{lem}
 \begin{proof}
From Lemma \ref{l124}, we know that $g\in C^\infty([0,L]^2)$ and has the following estimates
\begin{equation*}
 G^2_\alpha \leq \gamma_2 e^{\sigma_2\alpha^2}, \, \forall\  \alpha \geq \alpha_1,
\end{equation*}

Therefore, applying Lemma \ref{l12-2}, we get that any solution $u(\cdot)$ of the NSE in $\mathcal{A}$ satisfies the following estimates
	\begin{equation*}
		|A^{\frac{\alpha}{2}}u(t)| \leq M_{\alpha}\nu\ko^\alpha, \,\, \forall\  t\in \mathbb{R},\,\forall\  \alpha\in\mathbb{N},
	\end{equation*}
and for $\alpha>\alpha_1$, we have 
	\begin{equation*}
		M_{\alpha+1} := \frac{8\sqrt{2}}{\pi}\left(\frac{M^2_\alpha}{\delta_\alpha\nu\ko^2}+4G^2_{\alpha-1}+\sqrt{2}\Gamma_\alpha M^2_\alpha\right)^{\frac{1}{2}}.
\end{equation*}	
It follows that
	\begin{align*}
		M_{\alpha+1}^2 &\leq \frac{128}{\pi^2}\left(\frac{M^2_\alpha}{\delta_\alpha\nu\ko^2}+4\gamma_2e^{\sigma_2(\alpha-1)^2}+\sqrt{2}\Gamma_\alpha M_\alpha^2\right)\\
		&\leq \frac{128}{\pi^2}\left(\frac{M^2_\alpha}{\delta_\alpha\nu\ko^2}+\sqrt{2}\Gamma_\alpha M_\alpha^2\right)+\frac{2^7}{\pi^2}\gamma_2e^{\sigma_2(\alpha-1)^2}\\
		&= c_6 4^{\alpha+1} M^2_\alpha+c_7e^{\sigma_2(\alpha-1)^2},
	\end{align*}
where 
	\begin{align*} 
		c_6 :=\frac{128}{\pi^2}\left(\frac{1}{\delta_3\nu\ko^2}+2^2c_A^2M_1M_2+c_A\sqrt{M_1M_3} \right),\quad c_7 := \frac{2^7}{\pi^2}\gamma_2.
	\end{align*}
	Then we obtain the analogue of the relation (\ref{e1220})
	\begin{align*}
		M_\alpha^2 &\leq \gamma_3e^{\sigma_3\alpha^2},
	\end{align*}
	where
	\begin{align}
	\label{e1231}
		\sigma_3 &:= 2\ln 4+ 2\sigma_2,\\
		\label{e1232}
		\gamma_3 &:= (M^2_4+2c_7)c_6^{-4} e^{4\sigma_2}
						e^\frac{\ln c_6-4\ln 4-4\sigma_2}{2\sigma_3}.				  
	\end{align}
 \end{proof}

Now, we can pass to the proof of Theorem \ref{t121}.
\begin{proof}[proof of Theorem \ref{t121}]
	Let $u_0\in\mathcal{A}$, and $u(t), t\in\mathbb{R}$, be the solution of the NSE with $u(0)=u_0$.	
	From Lemma \ref{l126}, we obtain that
	\begin{equation*}
		\frac{|A^{\frac{\alpha}{2}}u(t)|^2}{\nu^2\ko^{2\alpha}}\leq M^2_\alpha\leq \gamma_3e^{\sigma_3\alpha^2}, \,\forall\  t\in\mathbb{R}, \, \forall\  \alpha\in\mathbb{N}, \alpha>\alpha_1.
	\end{equation*}	
	By choosing 
	\begin{equation*}
		\gamma_4 := \max\{\gamma_3, \frac{|A^{\frac{1}{2}}u_0|^2}{\nu^2\ko^2}, \cdots, \frac{|A^{\frac{\alpha_1}{2}}u_0|^2}{\nu^2\ko^{2\alpha_1}}\},
	\end{equation*}
	we infer that
	\begin{equation*}
		\frac{|A^{\frac{\alpha}{2}}u(t)|^2}{\nu^2\ko^{2\alpha}}\leq \gamma_4e^{\sigma_3\alpha^2}, \,\forall\  \alpha\in\mathbb{N},
	\end{equation*}
	i.e.
	\begin{equation*}
		u_0 \in \mathcal{C}(\sigma_3).
	\end{equation*}
	
	Since $u_0\in\mathcal{A}$ is arbitrary, the proof is complete.
\end{proof} 
\appendix
\section{}
Let $P_{\kappa}$ denote the orthogonal projection of $H$ onto
\begin{equation*}
\mathrm{span}\{w_{j} : Aw_{j}=\lambda_{j}w_{j},j\leq \kappa\}.
\end{equation*}

Let $g\in H, u_0\in \D(A^{\frac{1}{2}}), t_0\in\mathbb{R}$ and let $u(t)$ for $t\geq t_0$ be the solution of the NSE with ``initial data'' $u(t_0) = u_0$; moreover, for $\kappa\in\mathbb{N}$ we denote by $u_\kappa(t)\in P_\kappa H, t\geq t_0$ the solution of the following ODE
\begin{align}
\label{a1}
	\frac{du_\kappa(t)}{dt} +\nu Au_\kappa(t)+P_\kappa B(u_\kappa(t),  u_\kappa(t)) &= P_\kappa g, \quad t\geq t_0,\\\label{a2}
	u_\kappa(t_0) &= P_\kappa u_0,
\end{align}
usually called a Galerkin approximation of the solution $u(t)$ (e.g. see \cite{T95}, Chapter 2).
Standard ODE theory guarantees the existence of a unique solution $u_\kappa(t)$ for $\{t\in\mathbb{R}: |t-t_0|<\epsilon_0\}$ (where $\epsilon_0 = \epsilon_0(u_0)>0$) and then for all $t \geq t_0$ provided that for every $t_1\in [t_0, \infty)$, $u_\kappa(t)$ exists on $[t_0, t_1)$ and 
\begin{equation}
\label{a3}
	\sup_{t\in [t_0, t_1)} |u_\kappa(t)| < \infty.
\end{equation} 

The validity of the property (\ref{a3}) is obtained in the following way. From (\ref{a1}) we infer
\begin{equation*}
	\frac{1}{2}\frac{d}{dt}|u_\kappa(t)|^2+\nu|A^{\frac{1}{2}}u_\kappa(t)|^2 = (g, u_\kappa(t))\leq|g||u_\kappa(t)|, \quad t\in [t_0, t_1)
\end{equation*}
\begin{equation*}
\frac{1}{2} \frac{d}{dt} |A^{\frac{1}{2}}u_\kappa(t)|^2+\nu |Au_\kappa(t)|^2=(g,Au_\kappa)\leq\frac{|g|^2}{2\nu}+\frac{\nu|Au_\kappa|^2}{2},
\end{equation*}
and then 
\begin{equation*}
	\frac{d}{dt}|u_\kappa(t)|^2+\nu\ko^2|u_\kappa(t)|^2 \leq \frac{|g|^2}{\nu\ko^2}, \quad t\in [t_0, t_1),
\end{equation*}
\begin{equation*}
\frac{d}{dt} |A^{\frac{1}{2}}u_\kappa(t)|^2+\nu |Au_\kappa(t)|^2\leq\frac{|g|^2}{\nu}, 
\end{equation*}
from which it follows that
\begin{align*}
|u_\kappa(t)|^2
&\leq e^{-\nu\ko^2(t-t_0)}|u_0|^2+(1-e^{-\nu\ko^2(t-t_0)}) {\nu}^2G^2,
\end{align*}
\begin{equation}
\label{a5}
|A^{\frac{1}{2}}u_\kappa(t)|^2\leq e^{-\nu\ko^2(t-t_0)}|A^{\frac{1}{2}}u_0|^2+(1-e^{-\nu\ko^2(t-t_0)}) {\nu}^2\ko^2G^2, 
\end{equation}
for $t\geq t_0$.
In addition, it is not hard to prove that $u_\kappa(t)\to u(t)$ in $H$ uniformly for $t\in [t_0, t_1], \forall\  t_1\in\mathbb{R}$. 

The complexified version of  (\ref{a1}) and (\ref{a2}) has the following form
\begin{align} 
\label{a6}
	\frac{dV(\zeta)}{d\zeta} + \nu AV(\zeta)+P_\kappa B(V(\zeta), V(\zeta)) &= P_\kappa g, \\\label{aa7}
	V(t^0) &= P_\kappa V^0,
\end{align}
where $t_0\in\mathbb{R}$, $V(\zeta)\in P_\kappa H_\mathbb{C}$ and $V^0\in\D(A^{\frac{1}{2}})$. 
We will study the initial value problem for this equation.
The classical form of Cauchy's existence theorem (e.g. see \cite{JD73} Chapter 11, Section 5) ensures that the complex differential system (\ref{a6}) has a unique analytic solution $V(\zeta)$ defined in some neighborhood $\{\zeta\in\mathbb{C}: |\zeta-t^0|<\epsilon^0\}$ of $t^0$ satisfying the condition (\ref{aa7}). 

To extend the domain of existence for $V(\zeta)$ we proceed in the following manner (see also \cite{CF89}, \cite{FT79}); from (\ref{a6}) we obtain
\begin{equation*}
	\frac{d}{d\rho}V(t^0+\rho e^{i\theta})+e^{i\theta}[\nu AV(\zeta)+P_\kappa B(V(\zeta), V(\zeta))]|_{\zeta = t^0+\rho e^{i\theta}} = e^{i\theta} P_\kappa g,
\end{equation*}
where it is convenient to take $|\theta|\leq \frac{\pi}{4}$. Then proceeding as in the proof of Lemma \ref{lemma54}, we obtain that if 
\begin{equation*}
	\rho < \min\{\epsilon^0, \rho_1\},
\end{equation*}
where 
\begin{equation}
\label{eea14}
	\rho_1 =\rho_{\max}(G, \frac{|A^{\frac{1}{2}}V^0|}{\nu\ko}):= {\sqrt{2}}\left\{4\cdot 24^3 \cdot c_L^8\left[\frac{2^{\frac{1}{3}}}{24}G^2+\frac{|A^{\frac{1}{2}} V^0|^2}{\nu^2\ko^2}\right]^2\nu\ko^2\right\}^{-1},
\end{equation}
then 
\begin{equation*}
	|A^{\frac{1}{2}} V(\zeta)|^2 \leq \frac{2^{\frac{1}{3}}}{24}G^2(\nu\ko)^2 + \sqrt{2}|A^{\frac{1}{2}} V^0|^2.
\end{equation*}

Therefore if $\rho_1\geq\epsilon^0$, then defining $V(t^0+\epsilon^0 e^{i\theta})$ by
\begin{equation*}
	V(t^0+\epsilon^0 e^{i\theta}) = \lim_{\rho\to\epsilon^{0-}} V(t^0+\rho e^{i\theta})
\end{equation*}
where the limit exists and using again Cauchy's existence theorem, we can extend the analyticity domain of $V(\zeta)$ from $\{\zeta\in\mathbb{C}: |\zeta-t^0|<\epsilon^0\}$ to $\{\zeta\in\mathbb{C}: |\zeta-(t^0+\epsilon^0 e^{i\theta})|<\epsilon_1\}$ where $\epsilon_1>0$. Repeating this process for fixed $\theta$
, we can analytically extend $V(\zeta)$ to an open neighborhood of the segment $\{\zeta = t^0+re^{i\theta}: \epsilon\leq r \leq \rho_1\}$.
Since $\theta\in [-\frac{\pi}{4}, \frac{\pi}{4}]$ is arbitrary, we obtain an analytic extension of $V(\zeta)$ to the whole interior of the sector
\begin{equation}
\label{sector}
	\sector(t^0, \rho_1) := \{\zeta = t^0+\rho e^{i\theta}: |\theta|\leq \frac{\pi}{4}, \rho \in [0, \rho_1)\}
\end{equation}

Thus we obtain the following.
\begin{lem}
\label{la-1}
The solution $V(\zeta)$ of (\ref{a6}) satisfying (\ref{aa7}) exists, is $\D(A^{\frac{1}{2}})_\mathbb{C}$-valued analytic in the sector $\sector(t^0, \rho_1)$ and satisfies
\begin{equation}
\label{a14}
	|A^{\frac{1}{2}} V(\zeta)| \leq M_{1,1}\nu\ko, \quad \forall\  \zeta\in \sector(t^0, \rho_1),
\end{equation}
where 
\begin{equation}
\label{a18}
	M_{1,1} = M_1(G, \frac{|A^{\frac{1}{2}} V^0|}{\nu\ko}):= [\frac{2^{\frac{1}{3}}}{24}G^2 + \sqrt{2}\frac{|A^{\frac{1}{2}} V^0|^2}{(\nu\ko)^2}]^{\frac{1}{2}}.
\end{equation}
does not depend on $\kappa$ and depends only on $G$ and $\frac{|A^{\frac{1}{2}}V^0|}{(\nu\ko)^2}$.
\end{lem}

Clearly, in the case $t^0 \geq t_0, V(t^0) = u_\kappa(t^0)$, the restriction of $V(\zeta)$ to some interval $(t^0, t^1)$ coincides with the Galerkin approximation $u_\kappa(t)$ defined in (\ref{a1}) and (\ref{a2}). Therefore, letting the initial time $t^0$ vary over the whole $[t_0, \infty)$, we obtain the following.
\begin{cor}
\label{c2}
	For all $\kappa\in\mathbb{N}$, $u_\kappa(\zeta)$ is $\D(A^{\frac{1}{2}})_\mathbb{C}$-valued analytic analytic in 
	\begin{equation}
	\label{a19}
	\Pi(t_{0}, \rho_1) := \{\zeta \in\mathbb{C}: \Re(\zeta)\geq t_{0}, \Im(\zeta) \leq \min\{\Re(\zeta), \frac{\rho_1}{\sqrt{2}}\}\}
	\end{equation}
	and the relation 
\begin{equation*}
	|A^{\frac{1}{2}} u_\kappa(\zeta)| \leq [M_{1,1}^2+\sqrt{2}G^2]^{\frac{1}{2}}\nu\ko
\end{equation*}
holds for $\forall\  \zeta\in \Pi(t_{0}, \rho_1)$.
\end{cor}

We conclude our consideration by presenting the justification of the Remark \ref{r52}. For this purpose we need the following 
\begin{lem}
\label{la1}
Let $\alpha\in \mathbb{N}$ and $\mathcal{N}_\alpha$ be a domain containing an interval $(t_0, t_\alpha)\subset \mathbb{R}$. Furthermore, let $u_\kappa(\zeta)$ be the (complex) Galerkin approximation of solution $u(t)$ on $[t_0, +\infty)$ of the NSE satisfying $u(t_0) = u_0$. If each $u_\kappa$ is $P_\kappa H_\mathbb{C}$-valued analytic in $\mathcal{N}_\alpha$ such that
\begin{equation}
\label{a15}
	\sup\{|A^{\frac{\alpha}{2}}u_\kappa(\zeta)|, \zeta\in\mathcal{N}_\alpha\} \leq M_\alpha\nu\ko^\alpha <\infty,
\end{equation}
then $u(\zeta)$ is a $\D(A^{\frac{\alpha}{2}})_\mathbb{C}$-valued analytic function in $\mathcal{N}_\alpha$ and 
   \begin{equation*}
	\sup\{|A^{\frac{\alpha}{2}}u(\zeta)|, \zeta\in\mathcal{N}_\alpha\} \leq M_\alpha\nu\ko^\alpha.
\end{equation*}
\end{lem}
\begin{proof}
	Note that $\lim_{\kappa\to\infty}|u_\kappa(t)-u(t)|=0$ for all $t\geq t_0$. Therefore, for any $h\in H_\mathbb{C}$, $(u_\kappa(t), h) \to (u(t), h)$ for every $t\geq  t_0$. Applying Vitali's theorem and using (\ref{a15}) we obtain that $(u_\kappa(\zeta), h)$ is converging to a $\mathbb{C}$-valued analytic function $u_h(\zeta)$. It is easy to show that $u(\zeta, h)$ is antilinear in $h$ and that 
	\begin{equation*}
		|u(\zeta, h)| \leq M_{\alpha} \nu |h|, \quad \forall\  h\in H_\mathbb{C}.
\end{equation*}
Therefore the Riesz-Fr\'{e}chet theorem yields $V(\zeta)\in H_\mathbb{C}$ such that 
\begin{equation*}
	u(\zeta, h) = (V(\zeta), h), \quad \forall\  h\in H_\mathbb{C},  \forall\  \zeta \in \mathcal{N}_\alpha.
\end{equation*}		 

This shows that $V(\zeta)$ is weakly analytic in $\mathcal{N}_\alpha$ and therefore is also strongly analytic (i.e. $V(\zeta)$ is a $H_\mathbb{C}$-valued analytic function)  (e.g. see page 377, 399 \cite{BP77}; page 93 \cite{HP57}). Moreover for $h\in \D(A^{\frac{\alpha}{2}})_\mathbb{C}$ and $\zeta\in \mathcal{N}_\alpha$, we have that 
\begin{align*}
	|(V(\zeta), A^{\frac{\alpha}{2}}h)| = \lim_{\kappa\to\infty}|(u_\kappa(\zeta), A^{\frac{\alpha}{2}}h)|=\lim_{\kappa\to\infty}|(A^{\frac{\alpha}{2}}u_\kappa(\zeta), h)|\leq M_\alpha \nu\ko^\alpha |h|.
\end{align*}
Thus $V(\zeta)\in\D(A^{\frac{\alpha}{2}})_\mathbb{C}$ and $|A^{\frac{\alpha}{2}}V(\zeta)|\leq M_\alpha\nu\ko^\alpha$, since $A^{\frac{\alpha}{2}}$ is self-adjoint. Now 
\begin{equation*}
	(A^{\frac{\alpha}{2}}V(\zeta), h)=(V(\zeta), A^{\frac{\alpha}{2}}h),
\end{equation*}
for every $h\in\D(A^{\frac{\alpha}{2}})_\mathbb{C}$.
Since $\D(A^{\frac{\alpha}{2}})_\mathbb{C}$ is dense in $H_\mathbb{C}$, for any $h^0\in H_\mathbb{C}$ there exists a sequence $\{h^m\}\subset \D(A^{\frac{\alpha}{2}})$ such that $h^m\to h^0$, and hence
\begin{equation}
	|(A^{\frac{\alpha}{2}}V(\zeta), h^0)-(A^{\frac{\alpha}{2}}V(\zeta), h^m)| \leq M_\alpha\nu\ko^\alpha |h^0-h^m|
\end{equation}
which implies that $(A^{\frac{\alpha}{2}}V(\zeta), h^m)\to (A^{\frac{\alpha}{2}}V(\zeta), h^o)$ uniformly in $\zeta\in\mathcal{N}_\alpha$ and therefore $A^{\frac{\alpha}{2}}V(\zeta)$ is also weakly analytic. Hence $A^{\frac{\alpha}{2}}V(\zeta)$ is an $H_\mathbb{C}$-valued analytic function. 

For any $h\in\D(A^{\frac{1}{2}})$,  by using Ladyzhenskaya's inequality and the analyticity of $A^{\frac{\alpha}{2}}V(\zeta)$, we have that as $\xi\to 0$, 
\begin{align*}
	\frac{1}{\xi}((B&(V(\zeta+\xi), V(\zeta+\xi)), h)-(B(V(\zeta), V(\zeta)), h)) \\
	=&\frac{1}{\xi}[(B(V(\zeta+\xi)-V(\zeta), V(\zeta)), h)+(B(V(\zeta), V(\zeta+\xi)-V(\zeta)), h)\\
	&+(B(V(\zeta+\xi)-V(\zeta), V(\zeta+\xi)-V(\zeta)), h)]\\
	=&(B(\frac{V(\zeta+\xi)-V(\zeta)}{\xi}, V(\zeta)), h)+(B(V(\zeta), \frac{V(\zeta+\xi)-V(\zeta)}{\xi}), h)\\
	&+\xi(B(\frac{V(\zeta+\xi)-V(\zeta)}{\xi}, \frac{V(\zeta+\xi)-V(\zeta)}{\xi}), h)\\
	\to& (B(\frac{dV(\zeta)}{d\zeta}, V(\zeta)), h)+(B(V(\zeta), \frac{dV(\zeta)}{d\zeta}), h)\;,
\end{align*}
and hence $(B(V(\zeta), V(\zeta)), h)$ is analytic.

But for $\zeta = t\in(t_0, t_\alpha)$, $V(t) = u(t)$, so the following equation
\begin{equation*}
	(\frac{dV(\zeta)}{d\zeta}, h)+\nu (A^{\frac{1}{2}}V(\zeta), A^{\frac{1}{2}}h)+(B(V(\zeta), V(\zeta)), h) = (g, h), 
\end{equation*}
holds for $t\in(t_0, t_\alpha)$ and hence it holds for $\zeta\in \mathcal{N}_\alpha$ by analyticity. Now it follows that $V(\zeta)$ satisfies the complexified NSE (\ref{comnseq}) in $\mathcal{N}_\alpha$. We obtain, in particular, $V(\zeta) = u(\zeta)$ for all $\zeta\in \mathcal{N}_\alpha$.
\end{proof}

\begin{remark}
\label{ra4}
	It should be clear that, for the solution $u(\zeta)$ of the NSE (\ref{comnseq}) considered in Remark \ref{r52}, the differential inequalities involving $\frac{d}{d\rho}|A^{\frac{\alpha+1}{2}}u(t_0+\rho e^{i\theta})|^2$ used to establish estimates independent of $t_0, \rho$ and $\theta$ for $|A^{\frac{\alpha+1}{2}}u(\zeta)|$, are ``strikingly similar'' to those involving $\frac{d}{d\rho}|A^{\frac{\alpha+1}{2}}u_\kappa (t_0+\rho e^{i\theta})|^2$, where $u_\kappa(\zeta)$ is the solution of (\ref{a6}) satisfying (\ref{a2}). These estimates and the classical form of Cauchy's existence theorem imply the existence of $u_\kappa(\zeta)$ as well as the estimates given in the Proposition \ref{p10} and \ref{p11}. Thus, due to the considerations already made in this Appendix, the latter estimates imply both the existence of $u(\zeta)$ and the estimates of $|A^{\frac{\alpha+1}{2}}u(\zeta)|$ obtained by using the short procedure given in Remark \ref{r52}. Because of this fact, the latter estimates are usually referred as ``a priori estimates'' (e.g. see \cite{T95}, \cite{CF89}).
\end{remark}

\bibliographystyle{amsxport}
\bibliography{./ref}

\begin{bibdiv}
\begin{biblist}

\bib{BP77}{book}{
      author={Brown, A.},
      author={Pearcy, C.M.},
       title={Introduction to operator theory: Elements of functional
  analysis},
      series={Graduate texts in mathematics},
   publisher={Springer-Verlag},
        date={1977},
}

\bib{CHEN94}{article}{
      author={Chen, W.},
       title={{New a priori estimates in Gevrey class of regularity for weak
  solutions of 3D Navier-Stokes equations.}},
        date={1994},
     journal={Differ. Integral Equ.},
      volume={7},
      number={1},
       pages={101\ndash 107},
}

\bib{CF89}{book}{
      author={Constantin, P.},
      author={Foias, C.},
       title={{Navier-Stokes Equations}},
      series={Chicago Lectures in Mathematics},
   publisher={University of Chicago Press},
        date={1989},
        ISBN={9780226115498},
}

\bib{DFJ05}{article}{
      author={{Dascaliuc}, R.},
      author={{Foias}, C.},
      author={{Jolly}, M.~S.},
       title={{Relations Between Energy and Enstrophy on the Global Attractor
  of the 2-D Navier-Stokes Equations}},
        date={2005},
     journal={Journal of Dynamics and Differential Equations},
      volume={17},
       pages={643\ndash 736},
}

\bib{JD73}{book}{
      author={Dieudonn{\'e}, J.},
       title={Infinitesimal calculus},
   publisher={Hermann Press},
        date={1973},
}

\bib{ESS}{article}{
      author={{Escauriaza}, L.},
      author={{Seregin}, G.},
      author={V., {Sverak}},
       title={{$L^3,\infty$-solutions of Navier-Stokes equations and backward
  uniqueness, (Russian)}},
        date={2003},
     journal={Uspekhi Mat. Nauk},
      volume={58},
       pages={3\ndash 44},
        note={translation in {\it Russian Math. Surveys} {\bf 58} (2003), no.
  2, 211–-250},
}

\bib{FHN07}{article}{
      author={{Foias}, C.},
      author={{Hoang}, L.},
      author={{Nicolaenko}, B.},
       title={{On the helicity in 3-D periodic Navier-Stokes Equations I: The
  nonstatistical case}},
        date={2007},
     journal={Proc. London Math. Soc.},
      volume={94},
      number={1},
       pages={53\ndash 90},
}

\bib{FJL02}{article}{
      author={Foias, C},
      author={Jolly, M~S},
      author={Li, W-S},
       title={Nevanlinna-pick interpolation of attractors},
        date={2002},
     journal={Nonlinearity},
      volume={15},
      number={6},
       pages={1881},
}

\bib{FJY12}{article}{
      author={{Foias}, C.},
      author={{Jolly}, M.~S.},
      author={{Yang}, M.},
       title={{On Single Mode Forcing of the 2D-NSE}},
        date={2012},
     journal={Journal of Dynamics and Differential Equations},
}

\bib{FJK96}{article}{
      author={{Foias}, C.},
      author={{Jolly}, M.S.},
      author={{Kukavica}, I.},
       title={Localization of attractors by their analytic properties},
        date={1996},
     journal={Nonlinearity},
      volume={9},
      number={6},
       pages={15\ndash 65},
}

\bib{F01}{book}{
      author={{Foias}, C.},
      author={{Manley}, O.},
      author={{Rosa}, R.},
      author={{Temam}, R.},
       title={{Navier-Stokes Equations and Turbulence}},
      series={Encyclopedia of Mathematics and its Applications},
   publisher={Cambridge University Press},
        date={2001},
        ISBN={9780521360326},
}

\bib{FMT93}{article}{
      author={Foias, C.},
      author={Manley, O.~P.},
      author={Temam, R.},
       title={Bounds for the mean dissipation of {$2$}-{D} enstrophy and
  {$3$}-{D} energy in turbulent flows},
        date={1993},
        ISSN={0375-9601},
     journal={Phys. Lett. A},
      volume={174},
      number={3},
       pages={210\ndash 215},
}

\bib{FT79}{article}{
      author={{Foias}, C.},
      author={{Temam}, R.},
       title={{Some analytic and geometric properties of the solutions of the
  evolution Navier-Stokes equations}},
        date={1979},
     journal={Journal de mathematiques pures et appliquees},
      volume={58},
       pages={339\ndash 370},
}

\bib{FT94}{article}{
      author={{Foias}, C.},
      author={{Temam}, R.},
       title={Approximation of attractors by algebraic or analytic sets},
        date={1994},
     journal={SIAM J. Math. Analysis},
      volume={25},
      number={5},
       pages={1269Ð1302.},
}

\bib{HP57}{book}{
      author={Hille, E.},
      author={Phillips, R.S.},
       title={Functional analysis and semi-groups},
      series={American Mathematical Society: Colloquium publications},
   publisher={American Mathematical Society},
        date={1957},
        ISBN={9780821810316},
}

\bib{HT}{article}{
      author={{Hoff}, D.},
      author={{Tsyganov}, E.},
       title={{Time analyticity and backward uniqueness of weak solutions of
  the Navier–Stokes equations of multidimensional compressible flow}},
        date={2008},
     journal={J. Differential Equations},
      volume={245},
       pages={3068\ndash 3094},
}

\bib{Iooss}{article}{
      author={{Iooss}, G.},
       title={{Sur la deuxieme bifucation d’une solution stationaire de
  systems du type Navier-Stokes equations}},
        date={1977},
     journal={Arch. Rational Mech. Anal.},
      volume={64},
       pages={339\ndash 369},
}

\bib{Mas}{article}{
      author={{Masmuda}, K.},
       title={{On the analyticity and the unique continuation theorem for
  solutions of the Navier–Stokes equation}},
        date={1967},
     journal={Proc. Japan Acad.},
      volume={43},
       pages={827\ndash 832},
}

\bib{Dumas}{book}{
      author={(p{\`e}re), Alexandre~Dumas},
      author={Coward, David},
       title={The three musketeers},
      series={Oxford World's Classics},
   publisher={Oxford University Press, UK},
        date={1991},
}

\bib{R70}{book}{
      author={Rudin, W.},
       title={Real and complex analysis},
   publisher={McGraw-Hill},
        date={1970},
}

\bib{T95}{book}{
      author={Temam, R.},
       title={Navier-stokes equations and nonlinear functional analysis},
      series={CBMS-NSF Regional Conference Series in Applied Mathematical},
   publisher={Philadelphia : Society for Industrial and Applied Mathmatics},
        date={1983},
}

\bib{T97}{book}{
      author={Temam, R.},
       title={Infinite dimensonal dynamical systems in mechanics and physics},
      series={Applied Mathematical Sciences},
   publisher={Springer},
        date={1997},
}

\bib{Ts}{article}{
      author={{Tsyganov}, E.},
       title={{On time analyticity of weak solutions of the compressible
  Navier–Stokes equations}},
        date={2007},
     journal={Phys. D.},
      volume={221},
       pages={97\ndash 103},
}

\end{biblist}
\end{bibdiv}

\end{document}